\documentclass[a4paper,11pt, leqno]{article}
\usepackage{amsmath, amssymb, amsthm, amsfonts}
\usepackage{mathrsfs}
\usepackage{cases}

\theoremstyle{plain}
\newtheorem{theorem}{Theorem}

\newtheorem{corollary}{Corollary}
\newtheorem{lemma}{Lemma}

\theoremstyle{remark}
\newtheorem{remark}{Remark}

\renewcommand{\Re}{{\rm Re}}
\renewcommand{\Im}{{\rm Im\,}}
\renewcommand{\a}{\alpha}

\renewcommand{\d}{\displaystyle}

\DeclareMathOperator*{\Res}{Res}

\allowdisplaybreaks

\numberwithin{equation}{section}

\title{On the sum of $\Delta_k(n)$ in the Piltz divisor problem for $k=3$ and $k=4$}
\author{T. Makoto Minamide, Yoshio Tanigawa and Nigel Watt}
\date{\empty}

\begin{document}

\maketitle

\begin{abstract}
Let $\Delta_k(x)$ be the error term
in the classical asymptotic formula for the sum $\sum_{n \leq x}d_k(n)$, where $d_k(n)$ 
is the number of ways $n$ can be written as a product of $k$ factors. We study the analytic properties
of the Dirichlet series $\sum_{n=1}^{\infty}\Delta_k(n)n^{-s}$ and use Perron's formula to estimate
 the sums $\sum_{n\leq x} \Delta_3(n)$ and $\sum_{n\leq x} \Delta_4(n)$ for large $x>0$.   
\end{abstract}

\footnote[0]{ \hspace{-4mm}
 

This work is supported by JSPS KAKENHI No.~22K03245.
}

\section{Introduction}

Let $d_k(n)$ denote the number of ways that $n$ can be written as a product of $k$ factors. 
Thus $d_k(n)$ is the coefficient of $n^{-s}$ in the Dirichlet series for $\zeta^k(s)$, where $\zeta(s)$ 
is the Riemann zeta-function. Hereafter we shall write, as usual, $s=\sigma+it$ for a complex variable $s$.
Let $\Delta_k(x)$ be the error term of the summatory function $\sum_{n \leq x}d_k(n)$ defined by 
\begin{equation*}
\Delta_k(x)=\sum_{n \leq x}d_k(n)-xP_k(\log x),
\end{equation*}
where $P_k(u)$ is a polynomial  of degree $k-1$ in $u$ given by
\begin{equation} \label{Pk}
xP_k(\log x)=\Res_{s=1}\left(\zeta^k(s)\frac{x^s}{s}\right).
\end{equation}
The study of bounds for $\Delta_k(x)$ and the study of mean values of $\Delta_k(x)$ form two of the main objects 
of the theory of divisor problems.
In this paper we shall give new upper bounds for the sum of $\Delta_k(n)$ for $k=3$ and $4$.

Before stating our results, we provide some background on the problem.
The case $k=2$ is the classical Dirichlet divisor problem.
The error term $\Delta_2(x)$ is defined by
$$
\Delta_2(x)=\sum_{n \leq x}d_2(n)-x(\log x+2\gamma_0-1),
$$
where $\gamma_0$ is Euler's constant.  Dirichlet proved that $\Delta_2(x) \ll x^{1/2}$ 
using his famous hyperbola method. Dirichlet's upper bound has been improved by many authors, e.g.
Vorono\"{i} showed in 1903 that $\Delta_2(x) \ll x^{1/3} \log x$, and it is now known 
(due to Huxley \cite{Hu}) that $\Delta_2(x) \ll x^{131/416}(\log x)^{26947/8320}$.
It is conjectured that $\Delta_2(x) \ll x^{1/4+\varepsilon}$   
(note that here, and wherever it subsequently occurs, $\varepsilon$ denotes   
an arbitrarily small positive constant that is not necessarily the same at each occurrence).  

Vorono\"{i} \cite{V} proved that  
\begin{equation*} 
\int_1^x \Delta_2(u)du=\frac14x+\frac{x^{3/4}}{2\sqrt{2}\pi^2}\sum_{n=1}^{\infty}\frac{d_2(n)}{n^{5/4}}
\cos\left(4\pi\sqrt{nx}-\frac{\pi}{4}\right)+O(x^{1/4})
\end{equation*}
and 
\begin{equation*} 
\sum_{n \leq x} \Delta_2(n)=\frac12 x \log x+\left(\gamma_0-\frac14\right)x+O(x^{3/4}).
\end{equation*}
As for higher power moments of $\Delta_2(x)$, Tong \cite{To2} proved that 
$$
\int_1^X \Delta_2^2(x) dx=\frac{1}{6\pi^2}\sum_{n=1}^{\infty}\frac{d^2(n)}{n^{\frac32}}X^{\frac32}+O(X\log^5X).
$$
This error term was improved to $O(X\log^3 X\log\log X)$ by Lau and Tsang \cite{LT}. 
Moreover, the asymptotic formula  
$$
\int_1^X \Delta_2^m(x)dx=c_m X^{1+m/4}+O(X^{1+m/4-\delta_m+\varepsilon})
$$
is known for $3 \le m \leq 9$, where $c_m$ and  $\delta_m>0$ are explicitly given constants. 
The cases $m=3,4$ are due to Tsang \cite{Ts} and the cases $5 \leq m \leq 9$ are due to Zhai \cite{Z}. 

For the case $k=3$, it is known that 
$$
\int_1^X\Delta_3^2(x)dx=\frac{1}{10\pi^2}\sum_{n=1}^{\infty}\frac{d_3^2(n)}{n^{\frac43}}X^{\frac53}
+O(X^{\frac{14}{9}+\varepsilon}) 
$$
(see Tong \cite{To2}, Ivi\'c \cite[(13.43)]{I1} and Heath-Brown's Notes \cite[p.~178  and p.~327]{T}). 
Moreover Ivi\'{c} \cite[Theorem 13.10]{I1} studied 
higher power moments of $\Delta_3(x)$. In \cite{HB0}, Heath-Brown showed that 
$$
\int_1^X\Delta_4^2(x)dx \ll 
X^{7/4+\varepsilon}  
$$
(see also Ivi\'c \cite[Theorem 13.4]{I1}).  Besides these, there are extensive researches 
on the higher power moments of $\Delta_k(x)$ or its higher power moments 
in short intervals (see e.g. Ivi\'c and Zhai \cite{IZ} or Cao, Tanigawa and Zhai \cite{CTZ16}). 

In \cite{I2}, Ivi\'{c} studied the integral of $\Delta_k(x)$ ($k \geq 3$) and its mean square.
He adopted a different definition of the error term, which we write 
$$
\tilde{\Delta}_k(x)=\sum_{n \leq x}d_k(n)-\frac12d_k(x)-\Res_{s=1} \left(\zeta^k(s)\frac{x^s}{s}\right)-\zeta^k(0),
$$
where $d_k(x)=0$ if $x$ is not a positive integer. 
The relation 
between these two error terms is
$$
\tilde{\Delta}_k(x)=\Delta_k(x)-\frac12d_k(x)-\zeta^k(0).
$$
As for the integral of $\tilde{\Delta}_k(x)$, based on the deep results on higher power moments of 
the Riemann zeta-function in the critical strip, he showed that  
\begin{align*}
\int_1^x \tilde{\Delta}_k(u)du \ll x^{\theta_k+\varepsilon}      
\end{align*}
where $\theta_k \leq 3/2-1/k \ \ (3 \leq k \leq 8), \ \ \theta_9 \leq 145/102, \ \ \theta_{10} \leq 
29/20, \ \ \theta_{11} \leq 157/106, \ \ \theta_{12} \leq 3/2.$ 
\footnote{Using the truncated Vorono\"i formula for $\Delta_3(x)$ in Titchmarsh \cite[(12.4.6)]{T}, 
Ivi\'c asserted that $\int_1^x \tilde{\Delta}_3(u)du \ll x^{1+\varepsilon}$. However as Heath-Brown 
pointed out in \cite[p.\,409]{HB}, Titchmarsh missed a certain condition which is required in 
Atkinson's original paper \cite{A}. It seems that this condition was also overlooked by Ivi\'{c}, 
and that there is consequently a problem with the justification given in \cite[(1.7)]{I2}  
for his assertion concerning $\int_1^x \tilde{\Delta}_3(u)du$.
}

In this paper we consider the sum of $\Delta_k(n)$ ($k=3$ and $4$) through the Dirichlet series 
generated by these error terms, that is, we study the analytic properties of 
\begin{equation} \label{def-Dk}
D_k(s)=\sum_{n=1}^{\infty}\frac{\Delta_k(n)}{n^s}
\end{equation}
and derive the sum of $\Delta_k(n)$ by using Perron's formula. 
From the well-known bound $\Delta_k(x) \ll x^{(k-1)/(k+1)+\varepsilon}$ (see Titchmarsh \cite{T}), 
$D_k(s)$ is absolutely convergent in  the half plane $\sigma > \frac{2k}{k+1}$. 
For $k=3$ and $4$, it is known that  
\begin{equation} \label{D34-bound}
\Delta_k(x) \ll x^{1/2+\varepsilon}.
\end{equation}
For $\Delta_3(x)$, this upper bound was improved to
$$
\Delta_3(x) \ll x^{1/2-\delta_0+\varepsilon}
$$
with some positive number $\delta_0$. For example, $\delta_0=1/150$ by Atkinson (1941),  $\delta_0=1/22$ by
J.-R. Chen (1965), $\delta_0=5/96$ by Kolesnik (1981).  See Ivi\'c \cite[p.~381]{I1} for more details. 
However we shall use the bound in \eqref{D34-bound} in this paper,  since we would like to 
treat both cases in a unified way.

Our first theorem is the possibility of analytic continuation of $D_k(s)$ for any $k$. 
\begin{theorem} \label{thm1}
Let $D_k(s)$ be the function defined by \eqref{def-Dk}. Then $D_k(s)$ can be continued to the whole $s$-plane 
as a meromorphic function, whose poles are located at $s=1, 0 , -1 , \ldots$.  
\end{theorem}

The proof will be given in Section 2. This theorem is a generalization of Theorem 1 of \cite{FTZ}. 

\medskip

By modifying slightly the method of analytic continuation used in our proof of Theorem \ref{thm1}, 
we obtain (as shown in Sections 3--6) our second theorem.
\begin{theorem} \label{newthm}
Let $x>0$ be large and let $N$ be a positive integer. 
Then we have
\begin{align}
\sum_{n \leq x}\Delta_3(n)&=
\frac12 x P_3(\log x)+\frac{x}{2\pi^2\sqrt{3}}\sum_{n \leq N} \frac{d_3(n)}{n}\cos\left(6\pi(nx)^{\frac13}
+\frac{3\pi}{2}\right)-\frac18x  \label{thm2-3} \\
& \quad   +O\left(x^{\frac76+5\eta}N^{-\frac13}\left(1+\frac{N}{x}\right)\right)  \nonumber \\
\intertext{and}
\sum_{n \leq x} \Delta_4(n) & =\frac12 x P_4(\log x)  + \frac{x^{9/8}}{4\pi^2}\sum_{n \leq N} 
\frac{d_4(n)}{n^{7/8}}\cos\left(8\pi(nx)^{\frac14}+\frac{7\pi}{4}\right) \label{thm2-4} \\ 
& \quad {} + O\left(x^{\frac54+5\eta}N^{-\frac14}\left(1+\frac{N^{\frac34}}{x^{\frac54}}\right)\right)
         + O\left(xN^{\eta}\right),  \nonumber
\end{align}
where $\eta$ is any fixed constant such that $0<\eta<1/9$. 
\end{theorem}

The $O$-terms in \eqref{thm2-4} make the term $\frac12 x P_4(\log x)$ superfluous there.   
We prefer not to omit this term, since it is convenient for us 
that equations \eqref{thm2-3} and \eqref{thm2-4} have a similar shape. 
From Theorem~\ref{newthm} we deduce the following. 
\begin{corollary} \label{cor1} 
For any large $x >0$ we have
\begin{align} 
\sum_{n \leq x}\Delta_3(n) \ll x\log^3 x  \label{cor1-3}
\intertext{and}
\sum_{n \leq x}\Delta_4(n) \ll x^{7/6+\varepsilon}.  \label{cor1-4}
\end{align}
\end{corollary}

\medskip

By applying the theory of exponential sums to the sums in \eqref{thm2-3} and \eqref{thm2-4} 
we can improve the above Corollary \ref{cor1}.    
Let  
\begin{equation}\label{Def-rho-constant}
\varrho := \inf\left\{ k\in [0,{\textstyle\frac12}] \,:\, \mathscr{G}\ni(k , k+{\textstyle\frac12})\right\}\;, 
\end{equation} 
where $\mathscr{G}\subset [0,\frac12]\times [\frac12,1]$ is the set of all exponent pairs, known or unknown: 
details of the theory of exponent pairs and the associated 
$A$ and $B$ processes can be found in Graham and Kolesnik \cite{GK} (or see \cite[Note~5.20]{T}). 
A long-standing and very consequential conjecture is that $\varrho = 0$  
(as noted in \cite[p.~214]{IK}, 
implications that this `exponent pair hypothesis'  
would have include the conjectured bound $\Delta_2(x)\ll x^{1/4+\varepsilon}$ 
and Lindel\"of's hypothesis that $\zeta(\frac12 + it)\ll t^{\varepsilon}$ for $t\geq 1$).  
Bourgain \cite[Theorem 6]{Bo} has shown that 
\begin{equation}\label{BourgainEP}
\mathscr{G}\ni\left({\textstyle\frac{13}{84}}+\varepsilon , {\textstyle\frac{55}{84}}+\varepsilon \right) 
\quad\,\text{for $\,{\textstyle 0<\varepsilon\leq\frac{29}{84}}$} , 
\end{equation} 
so it is certainly the case that  
\begin{equation}\label{rho_LEQ_Bo} 
0\leq\varrho\leq{\textstyle\frac{13}{84}}\;.
\end{equation}
Our improvement of Corollary \ref{cor1} is the following. 
\begin{theorem} \label{thm3} For any large $x>0$ we have
\begin{align}
\sum_{n \leq x} \Delta_3(n) & \ll x(\log x)^{\frac83} \label{thm3-3}
\intertext{and}
\sum_{n \leq x} \Delta_4(n) & \ll x^{\frac{9+10\varrho\vphantom{_t}}{8+8\varrho\vphantom{_t}}+\varepsilon}. \label{thm3-4}
\end{align}
\end{theorem}

\medskip 

The proof will be given in Section~8. 
Note that it follows immediately from \eqref{thm3-4} and \eqref{rho_LEQ_Bo} that  
\begin{equation}\label{KL-app-1}
\sum_{n \leq x} \Delta_4(n)\ll x^{\frac{443}{388}+\varepsilon} = x^{\frac98 + \frac{13}{776} +\varepsilon}\;.
\end{equation} 
Obviously, if the exponent pair hypothesis is correct then the exponent on the right-hand 
side of \eqref{thm3-4} is just $\frac98 + \varepsilon$. 
As for the integral of $\Delta_k(x)$, we derive from Theorem~\ref{thm3} 
and \eqref{rho_LEQ_Bo} the following. 
\begin{corollary} \label{cor2}
For any large $x>0$ we have
\begin{align*}
&\int_1^x \Delta_3(t)dt \ll x(\log x)^{\frac83},  \\  
\intertext{and}
&\int_1^x \Delta_4(t)dt \ll x^{\frac{9+10\varrho\vphantom{_t}}{8+8\varrho\vphantom{_t}}+\varepsilon} 
\ll x^{\frac{443}{388}+\varepsilon}.  
\end{align*}
\end{corollary}

\medskip 

The method we use to derive Corollary \ref{cor2} should be compared with Segal's 
identity in Lemma \ref{Segal} below: see the relevant discussion in Section~9. 

We remark, lastly, that both  \eqref{KL-app-1} and the 
similar bound for $\int_1^x \Delta_4(t)dt$ implicit in Corollary \ref{cor2}  
are critically dependent on Bourgain's exponent pair (in \eqref{BourgainEP}),
and that by also utilising certain exponent pairs 
found recently by Trudgian and Yang \cite{TY}  
we can improve (slightly) on \eqref{KL-app-1}, and get:   
\begin{equation}\label{Best-by-newEPs}
\sum_{n \leq x} \Delta_4(n)\ll 
x^{\frac98 + \frac{2471}{147580}+\varepsilon} = 
x^{\frac{336997}{295160}+\varepsilon}  
\end{equation} 
(note that $\frac{2471}{147580}\approx 1.6743\times 10^{-2}$, 
while $\frac{13}{776}\approx 1.6753\times 10^{-2}$). 
We obtain, as a corollary,  a similar  bound 
for the integral $\int_1^x \Delta_4(t)dt$. 
At the end of Section~8 we provide 
some details of our proof of \eqref{Best-by-newEPs} 
and its corollary.

\section{Analytic continuation of $D_k(s)$} 

We shall prove Theorem \ref{thm1}. Suppose that $\sigma>2$ and consider the sum 
\begin{equation*}
S_k=\sum_{n=0}^{\infty}\sum_{m=1}^{\infty}\frac{d_k(m)}{(n+m)^s},
\end{equation*}
which is absolutely convergent in this region. 
Rearranging the above sum by putting $l=n+m$, we have
\begin{align*}
S_k&=\sum_{l=1}^{\infty}\left(\sum_{m=1}^{l}d_k(m)\right)l^{-s} 
    =\sum_{l=1}^{\infty}\frac{lP_k(\log l)}{l^s}+\sum_{l=1}^{\infty}\frac{\Delta_k(l)}{l^s} \\
   &=\sum_{l=1}^{\infty}\frac{P_k(\log l)}{l^{s-1}}+D_k(s), 
\end{align*}
where $D_k(s)$ is defined by \eqref{def-Dk}. 
If we write $P_k(u)=a_{k-1}u^{k-1}+\cdots+a_1u+a_0$, then
\begin{equation} \label{hyouji-1}
S_k=\sum_{j=0}^{k-1}a_j(-1)^j \zeta^{(j)}(s-1)+D_k(s).
\end{equation}

Next we shall derive a different expression for $S_k$. For this purpose we apply the following lemma, 
a proof of which can be found in Andrews, Askey and Roy \cite[p.~81]{AAR}.
\begin{lemma}
Let $x>0$ and $-\sigma < c<0$. Then we have
$$
\Gamma(s)(1+x)^{-s}=\frac{1}{2\pi i}\int_{(c)}\Gamma(s+z)\Gamma(-z)x^z dz,
$$
where $\int_{(c)}$ means $\int_{c-i\infty}^{c+i\infty}$.
\end{lemma}

From this lemma, we see easily that 
\begin{align*}
(n+m)^{-s}&=n^{-s}\left(1+\frac{m}{n}\right)^{s} \\
          &=\frac{1}{2\pi i} \int_{(c)}\frac{\Gamma(s+z)\Gamma(-z)}{\Gamma(s)}n^{-s-z}m^z dz
\end{align*}
for $n \neq 0$. Assuming $\sigma>2$ and $1-\sigma<c<-1$, we find that 
\begin{align}
S_k&=\sum_{m=1}^{\infty}\frac{d_k(m)}{m^s}+\sum_{n=1}^{\infty}\sum_{m=1}^{\infty}\frac{d_k(m)}{(n+m)^s} \label{start} \\
   &=\zeta^k(s)+\frac{1}{2\pi i}\int_{(c)}\frac{\Gamma(s+z)\Gamma(-z)}{\Gamma(s)}\zeta(s+z)\zeta^k(-z)dz \notag \\
   &=\zeta^k(s)+J_{c}(s), \notag
\end{align}
say.

Let $N$ be any positive integer and $0<\varepsilon<1/2$.  
To get the analytic continuation of $J_{c}(s)$, we move the line of integration from $(c)$ to $(N+\varepsilon)$.
The integrand has poles at $z=-1, \ 0$ and odd positive integers $j$. Writing 
\begin{equation} \label{ryuusu}
R_{u}(s)=\Res_{z=u}\left(\frac{\Gamma(s+z)\Gamma(-z)}{\Gamma(s)}\zeta(s+z)\zeta^k(-z)\right),
\end{equation}
we get
\begin{equation} \label{hyouji-2}
S_k=\zeta^k(s)-R_{-1}(s)-R_{0}(s)-\sum_{\substack{j=1 \\ j {\rm :odd}}}^N R_{j}(s) +J_{N+\varepsilon}(s).
\end{equation}
From \eqref{hyouji-1} and \eqref{hyouji-2} we have
\begin{align} \label{D-hyouji}
D_k(s)&=-\sum_{j=0}^{k-1}a_j(-1)^j\zeta^{(j)}(s-1)+\zeta^k(s)-R_{-1}(s)  \\
      & \quad   -R_0(s)-\sum_{\substack{j=1 \\ j {\rm :odd}}}^N R_{j}(s) +J_{N+\varepsilon}(s). \notag
\end{align}
Here $R_0(s)=-\left(-\frac12\right)^k \zeta(s)$ has a pole of order 1 at $s=1$ and 
$R_j(s)=-(-1)^j \binom{s+j-1}{j}\zeta(s+j)\zeta^k(-j)$
is holomorphic. On the other hand, $R_{-1}(s)$ is a linear combination of products of 
$\Gamma^{(j_1)}(s-1)$ and $\zeta^{(j_2)}(s-1)$ with $j_1+j_2 \leq k-1$,
so $R_{-1}(s)$ has poles of order at most $k$ at integers less than or equal to $2$. 
The last integral $J_{N+\varepsilon}(s)$ in \eqref{D-hyouji}
can be continued as a holomorphic function to $\sigma >1-N-\varepsilon$. Since $N$ is arbitrary, 
$D_k(s)$ can be continued to the whole plane. Moreover $D_k(s)$ converges absolutely 
in the half plane $\sigma>\frac{2k}{k+1}$. Hence $s=2$ is not a pole of $D_k(s)$. This means 
that in the right-hand side of \eqref{D-hyouji} the first term and the term $-R_{-1}(s)$   
have principal parts at $s=2$ that cancel out.


\section{Expressions for $D_3(s)$ and $D_4(s)$}

In what follows, we consider the cases $k=3$ and $k=4$. As has already been shown,  
$D_k(s)$ has an analytic continuation to the whole $s$-plane (with some poles, but no other type of singularity). 
The relevant expression \eqref{D-hyouji} is, however, not adequate for the proof of Theorem \ref{newthm}.
We therefore reconsider the process of analytic continuation of $D_k(s)$ in a different manner. 

The polynomials $P_k(u)$ defined by \eqref{Pk} are given explicitly  by 
\begin{align}  
P_3(u)&=\frac12 u^2+(3\gamma_0-1) u+3\gamma_1+3\gamma_0^2-3\gamma_0+1  \label{shukou} \\
\intertext{and} 
P_4(u)&=\frac16 u^3+\left(2\gamma_0-\frac12\right)u^2+\left(4\gamma_1+6\gamma_0^2-4\gamma_0+1\right)u \label{shukou4} \\
&\quad {}+4\left(\gamma_2+3\gamma_0\gamma_1+\gamma_0^3\right)-\left(4\gamma_1+6\gamma_0^2\right)+4\gamma_0-1. \nonumber
\end{align}
Here $\gamma_j$ are the coefficients of the Laurent expansion of $\zeta(s)$ at $s=1$:
$\zeta(s)=1/(s-1)+\sum_{j=0}^{\infty}\gamma_j(s-1)^j$. 
As mentioned in the Introduction, the bound $\Delta_k(x) \ll x^{1/2+\varepsilon}$ is the only estimate for $|\Delta_k(x)|$ 
that we shall use, although a slightly stronger bound is known to hold in  the case $k=3$. 

Let $\eta$ be any small fixed positive number less than $1/9$. In \eqref{start} we move the line of integration 
from $(c)$ to $(-1/2-2\eta)$. Then we get 
\begin{equation*}
S_k=\zeta^k(s)-R_{-1}(s)+J_{-1/2-2\eta}(s),
\end{equation*}
where $J_{-1/2-2\eta}(s)$ and  $R_{-1}(s)$ are defined by \eqref{start} and \eqref{ryuusu}, respectively.
Recall that 
$$
J_{-1/2-2\eta}(s)=\frac{1}{2\pi i}\int_{(-1/2-2\eta)}\frac{\Gamma(s+z)\Gamma(-z)}{\Gamma(s)}\zeta(s+z)\zeta^k(-z)dz
$$
and
$$
R_{-1}(s)=\Res_{z=-1}\left(\frac{\Gamma(s+z)\Gamma(-z)}{\Gamma(s)}\zeta(s+z)\zeta^k(-z)\right).
$$
The explicit forms of $R_{-1}(s)$ are 
\begin{align} \label{R3}
&R_{-1}(s) \\[1ex]
&=\frac{\Gamma(s-1)}{\Gamma(s)}\left(-\left(3\gamma_1+\frac{\gamma_0^2}{2}+\frac{\pi^2}{12}\right)\zeta(s-1)
        +2\gamma_0\zeta'(s-1)-\frac12\zeta''(s-1)\right) \nonumber \\[1ex]
&\quad +\frac{\Gamma'(s-1)}{\Gamma(s)}\Bigl(2\gamma_0\zeta(s-1)-\zeta'(s-1)\Bigr) 
       -\frac{\Gamma''(s-1)}{\Gamma(s)}\cdot \frac{\zeta(s-1)}{2} \nonumber  
\end{align}
for $k=3$ and 
\begin{align}\label{R4}
R_{-1}(s)
&=\frac{\Gamma(s-1)}{\Gamma(s)}\left\{-\left(4\gamma_2+8\gamma_0\gamma_1-\frac{1}{6}\gamma_0^3+\frac{\pi^2}{4}\gamma_0
-\frac{\zeta(3)}{3}\right)\zeta(s-1) \right.  \\
& \quad \left. +\left(4\gamma_1+\frac52\gamma_0^2+\frac{\pi^2}{12}\right)\zeta'(s-1)-\frac32\gamma_0\zeta''(s-1)
+\frac16\zeta'''(s-1)\right\}  \nonumber \\
&\quad +\frac{\Gamma'(s-1)}{\Gamma(s)}\left\{\left(4\gamma_1+\frac52\gamma_0^2+\frac{\pi^2}{12}\right)\zeta(s-1)
 - 3\gamma_0\zeta'(s-1)+\frac12\zeta''(s-1)\right\} \nonumber \\
& \quad +\frac{\Gamma''(s-1)}{\Gamma(s)}\left\{-\frac32\gamma_0\zeta(s-1)+\frac12\zeta'(s-1)\right\} \nonumber \\
& \quad +\frac{\Gamma'''(s-1)}{\Gamma(s)}\cdot \frac16 \zeta(s-1) \nonumber
\end{align}
for $k=4$. Here we have used the well-known identities $\Gamma''(1)=\gamma_0^2+\frac{\pi^2}{6}$ and 
$\Gamma'''(1)=-2\zeta(3)-\gamma_0^3-\frac12\pi^2 \gamma_0.$

As a function of $s$, $J_{-1/2-2\eta}(s)$ is holomorphic in the half plane $\sigma>3/2+2\eta$. 
When we continue this function to the region $\sigma<3/2+2\eta$, the residue at $z=1-s$ 
$$
\Res\limits_{z=1-s}\left(\frac{\Gamma(s+z)\Gamma(-z)}{\Gamma(s)}\zeta(s+z)\zeta^k(-z)\right)=
\frac{\zeta^k(s-1)}{s-1}
$$ 
appears. Thus we get the expression
\begin{equation} \label{newhyouji}
S_k=\zeta^k(s)-R_{-1}(s)+\frac{\zeta^k(s-1)}{s-1}+J_{-1/2-2\eta}(s).
\end{equation}
In fact, the expression \eqref{newhyouji} is valid in the range $1/2+2\eta < \sigma < 3/2+2\eta$,
since $J_{-1/2-2\eta}(s)$ is holomorphic there. 
For convenience, define the function $F(s)$ by
\begin{equation} \label{F-def}
F(s)=-\sum_{j=0}^{k-1}a_j(-1)^j\zeta^{(j)}(s-1)+J_{-1/2-2\eta}(s).
\end{equation} 
Then, by \eqref{hyouji-1} and \eqref{newhyouji}, we have 
\begin{align} \label{D3-hyouji-0}
D_k(s)&=F(s) +\zeta^k(s)- R_{-1}(s)+\frac{\zeta^k(s-1)}{s-1} 
\end{align}
for all $s$ in the strip  $1/2+2\eta < \sigma < 3/2+2\eta.$

 
\section{The sum of $\Delta_k(n)$} 

The Dirichlet series $D_k(s)$ converges absolutely in $\sigma>3/2$ (in fact in $\sigma>3/2-\delta_0$ for $k=3$) 
and has an analytic continuation to the whole complex plane, with poles at $s=1, 0$ 
and the negative integers (but no other singularities).  
To estimate the sum $\sum_{n \leq x}\Delta_k(n)$ we apply Perron's formula to the Dirichlet series $D_k(s)$. 
Since the exact values of the abscissae of absolute convergence are not known 
we shall apply \cite[Theorem II.2.3]{Te}.

\begin{lemma}[Tenenbaum \cite{Te}]  \label{Perron}
Let $f(s)=\sum_{n \geq 1} a_n/n^s$ be a Dirichlet series with a finite abscissa of absolute convergence $\sigma_a$.
For $\kappa>\max\{0,\sigma_a\}, \ T \geq 1$ and $x \geq 1$, we have
\begin{align} \label{perron-0}
\sum_{n \leq x} a_n =\frac{1}{2\pi i} \int_{\kappa-iT}^{\kappa+iT} f(s) \frac{x^s}{s}ds+
O\left(x^{\kappa}\sum_{n =1}^{\infty} \frac{|a_n|}{n^{\kappa}(1+T|\log(x/n)|)}\right).
\end{align} 
\end{lemma}

\medskip

Suppose that $|a_n| \ll n^{\a}$ for all $n$. Then, by dividing the sum into the cases 
$n < \frac{x}{2}, \frac{x}{2} \leq n \leq x-2, x-2<n<x+2, x+2\leq n \le 2x $ and $2x<n$, 
it is easy to see that the error term in \eqref{perron-0} is bounded by
\begin{equation} \label{perron-1}
O\left(\frac{x^{\kappa}}{T}+\frac{x^{\a+1}}{T}\log x + x^{\a}\right).
\end{equation}

In our case $f(s)=D_k(s)$, we can take $\a=1/2+\eta/2, \kappa=3/2+\eta$  (where $\eta$  is a 
fixed small positive number), and from \eqref{perron-0} and \eqref{perron-1} we get
\begin{equation*}
\sum_{n \leq x}\Delta_k(n) = \frac{1}{2\pi i} \int_{\kappa-iT}^{\kappa+iT}D_k(s)\frac{x^s}{s}ds
          +O\left(x^{1/2+\eta}+\frac{x^{3/2+\eta}}{T}\right)
\end{equation*}
for any $x \geq 1 $ and $T \geq 1$. 
Substituting the expression \eqref{D3-hyouji-0} into the above formula we get
\begin{align}  
\sum_{n \leq x}\Delta_k(n) 
&=U_1+U_2-U_3+U_4 +O\left(x^{1/2+\eta}+\frac{x^{3/2+\eta}}{T}\right), \label{Deltakwa}
\end{align}
where 
\begin{align*}
U_1 =\frac{1}{2\pi i} \int_{\kappa-iT}^{\kappa+iT}F(s)\frac{x^s}{s}ds, &\quad  U_3 =\frac{1}{2\pi i} \int_{\kappa-iT}^{\kappa+iT}R_{-1}(s)\frac{x^s}{s}ds, \\[1ex]
U_2 =\frac{1}{2\pi i} \int_{\kappa-iT}^{\kappa+iT}\zeta^k(s)\frac{x^s}{s}ds, &\quad U_4=\frac{1}{2\pi i} \int_{\kappa-iT}^{\kappa+iT}\frac{\zeta^k(s-1)}{s-1}\frac{x^s}{s}ds. \\[1ex]
\end{align*}


\section{Estimation of $U_j$ \ ($1 \leq j \leq 4$)}  

\subsection{Analytic continuation and estimation of $J_{-1/2-2\eta}(s)$}  \label{sec5}

In order to estimate $U_1$ we need the analytic continuation of $F(s)$, 
especially of  $J_{-1/2-2\eta}(s)$, into a wider region. To make the process clear, we temporary write 
\begin{equation*} 
F_1(s)=J_{-1/2-2\eta}(s),  
\end{equation*}
which is holomorphic in $1/2+2\eta<\sigma<3/2+2\eta$.
Let $1/2+2\eta<\sigma < 1/2+4\eta$. 
Recalling (from \eqref{start}) the definition of $J_c(s)$ as a contour integral, 
we move the line of integration  from $(-1/2-2\eta)$ to $(-1/2-4\eta)$. 
Then the residue at $z=-s$ 
$$
\Res_{z=-s} \left(\frac{\Gamma(s+z)\Gamma(-z)}{\Gamma(s)}\zeta(s+z)\zeta^k(-z)\right)=-\frac12\zeta^k(s)
$$
appears. Therefore $F_1(s)=J_{-1/2-4\eta}(s)-\frac12\zeta^k(s)$ in the vertical strip $1/2+2\eta<\sigma<1/2+4\eta$.
But the function $J_{-1/2-4\eta}(s)$ here can be continued as a holomorphic function to the left 
beyond the line $\sigma=1/2+2\eta$. In fact, it is holomorphic in $-1/2+4\eta<\sigma< 1/2+4\eta$. 
Hence, by the identity theorem for analytic functions, we see that 
\begin{equation*} 
F_1(s)=J_{-1/2-4\eta}(s)-\frac12\zeta^k(s), \quad  \left(-\frac12+4\eta < \sigma < \frac12+4\eta\right).
\end{equation*}
In this way we get the following lemma. 
\begin{lemma} \label{lemma3}
Let 
$$
F_0(s)=-\sum_{j=0}^{k-1}a_j(-1)^j\zeta^{(j)}(s-1).
$$ Then we have
\begin{align*}
F(s)=\begin{cases}  \d F_0(s)+J_{-1/2-2\eta}(s) & 
\text{if \; $\frac12+2\eta<\sigma<\frac32+2\eta$} , \\ 
                    \d F_0(s)+J_{-1/2-4\eta}(s)-\frac12\zeta^k(s) & 
\text{if \; $-\frac12+4\eta<\sigma<\frac12+4\eta$} .   
     \end{cases} 
\end{align*}
\end{lemma}
We note that Lemma \ref{lemma3} provides an analytic continuation of $F(s)$ as a holomorphic function over the strip
$-1/2+4\eta<\sigma < 3/2+2\eta$.

Now we are going to derive a bound for the function $F(s)$.  First we show that
\begin{equation} \label{order1}
F\left(\frac32+\eta+it\right) \ll 1. 
\end{equation}
To see this, it is enough to show that the other four terms in \eqref{D3-hyouji-0} are bounded
on the line $\sigma=3/2+\eta$. 
It is clear that $D_k(s) \ll 1$ there, since the series in \eqref{def-Dk} converges absolutely 
for $\sigma > \frac32$ (when $k\in\{ 3,4\}$). Recalling the expressions in \eqref{R3} and \eqref{R4}, 
and noting that $\Gamma^{(j)}(s-1)/\Gamma(s) \ll (\log |t|)^j/|t|$ 
and $\zeta^{(j)}(s-1) \ll |t|^{1/6}$ ($0 \leq j \leq 3$, $\sigma=3/2+\eta$, $|t|\geq 1$),  
we see that $R_{-1}(s) \ll 1$ on the same line. 
The terms $\zeta^k(s)$ and $\zeta^k(s-1)/(s-1)$  in \eqref{D3-hyouji-0} are also bounded   
on this line (for $k=3$ or $4$). Hence we get \eqref{order1}.
We got this without using the integral form of $J_{-1/2-2\eta}(s)$. 

Next we shall estimate $J_{-1/2-2\eta}(s)$ on the line $\sigma=1/2+3\eta$.
We may assume $t \geq 1$. The form of $J_{-1/2-2\eta}(s)$ is
\begin{align} \label{Jform}
J_{-1/2-2\eta}(s)&=\frac{1}{2\pi \Gamma(s)}\int_{-\infty}^{\infty}\Gamma(\eta+i(t+v))\Gamma\left(\frac12+2\eta-iv\right)  \\
& \hspace{2.5cm} \times \zeta(\eta+i(t+v))\zeta^k\left(\frac12+2\eta-iv\right)dv. \notag
\end{align} 
From Stirling's formula and \cite[Theorem 1.9 and (7.58)]{I1}, we see that 
\begin{align*}
\Gamma(\eta+i(t+v)) & \ll (1+|t+v|)^{\eta-1/2}e^{-\frac{\pi}{2}|t+v|}, \\
\Gamma\left(\frac12+2\eta-iv\right) & \ll (1+|v|)^{2\eta} e^{-\frac{\pi}{2}|v|}, \\
\zeta(\eta+i(t+v)) & \ll (1+|t+v|)^{\frac12-\frac23\eta}, \\[1ex]
\Gamma\left(\frac12+3\eta+it \right) & \sim \sqrt{2\pi} t^{3\eta}e^{-\frac{\pi}{2}t}.
\end{align*}
To evaluate the integral in \eqref{Jform} we first split it up, writing  
$$
\int_{-\infty}^{\infty}=\int_{-\infty}^{-t}+ \int_{-t}^{0}+\int_0^{t}+\int_t^{\infty}=:I_1+I_2+I_3+I_4, 
$$ 
say, 
and then make use of the fact that, by Titchmarsh \cite[Theorem 7.5]{T} and (when $k=3$) H\"older's inequality, 
one has $\int_0^T |\zeta(\sigma+it)|^kdt \ll T$ for any fixed $\sigma\in(\frac12,1)$ and $k=3$ or $4$. 

(i) The integral $I_2$.  Since $-t \leq v \leq 0$, we have
\begin{align*}
I_2 & \ll \int_{-t}^0 (1+t+v)^{\frac13\eta} (1+|v|)^{2\eta} e^{-\frac{\pi}{2}t}\left|\zeta\left(\frac12+2\eta-iv\right)\right|^k dv \\
    & \ll t^{1+\frac73\eta}e^{-\frac{\pi}{2}t}.
\end{align*}

(ii) The integral 
$I_1$. Since $v\leq -t$, we have, replacing $v$ by $-v$,
\begin{align*}
I_1 \ll e^{\frac{\pi}{2}t} \int_{t}^{\infty}(1+v-t)^{\frac13\eta} v^{2\eta} e^{-\pi v}
\left|\zeta\left(\frac12+2\eta+iv\right)\right|^k dv.
\end{align*}
Let 
$$
H(T)=\int_T^{2T}(1+v-t)^{\frac13\eta}v^{2\eta}e^{-\pi v}\left|\zeta\left(\frac12+2\eta+iv\right)\right|^k dv
$$
for $T \geq t$. We have easily that
$$
H(T) \leq C (2T)^{1+\frac73\eta}e^{-\pi T}
$$
with an absolute constant $C>0$. Hence
\begin{align*}
I_1 & \ll e^{\frac{\pi}{2}t}\sum_{j=0}^{\infty}H(2^j t) \\
    & \ll e^{\frac{\pi}{2}t} (2t)^{1+\frac73\eta}\left(e^{-\pi t}+2^{1+\frac73\eta}e^{-2\pi t}
              +2^{2(1+\frac73\eta)}e^{-2^2\pi t}+\cdots\right) \\
    & \ll t^{1+\frac73\eta}e^{-\frac{\pi}{2} t},
\end{align*}
where the implied constants are absolute.

(iii) The integral 
$I_3$. We have
\begin{align*}
I_3 & \ll e^{-\frac{\pi}{2}t} \int_0^t (1+t+v)^{\frac13\eta}(1+v)^{2\eta} e^{-\pi v}\left|
\zeta\left(\frac12+2\eta-iv\right)\right|^k dv \\[1ex]
& \ll t^{1+\frac13\eta}e^{-\frac{\pi}{2}t}.
\end{align*}

(iv) The integral 
$I_4$. Similarly to (ii) above, we have
\begin{align*}
I_4  & \ll e^{-\frac{\pi}{2}t} \int_t^{\infty}(1+t+v)^{\frac13\eta}v^{2\eta}e^{-\pi v}
          \left|\zeta\left(\frac12+2\eta+iv\right)\right|^k dv \\
     & \ll e^{-\frac{\pi}{2}t} \int_t^{\infty}(1+2v)^{\frac13\eta}v^{2\eta}e^{-\pi v}
          \left|\zeta\left(\frac12+2\eta+iv\right)\right|^k dv  \\
     & \ll  t^{1+\frac73\eta}e^{-\frac{3\pi}{2}t}.
\end{align*}
From these estimates we have for $|t|\gg 1$
\begin{equation} \label{order2}
J_{-1/2-2\eta}\left(\frac12+3\eta+it\right) \ll |t|^{1-\frac23\eta}.
\end{equation}
As for the function $F(s)$, 
it follows from Lemma \ref{lemma3}, \eqref{order2} and the bounds 
$ \zeta^{(j)}\left(-\frac12+3\eta+it \right) \ll |t|^{1-2\eta}$ ($0 \leq j \leq 3$) 
that we have 
\begin{equation} \label{order-F}
F\left(\frac12+3\eta+it\right) \ll |t|^{1-\frac23\eta}.
\end{equation}
Using \eqref{order1}, \eqref{order-F} and the convexity theorem, we find that   
\begin{equation} \label{order-F-kukan}
F(\sigma +it) \ll |t|^{-\frac{1-2\eta/3}{1-2\eta}(\sigma-3/2-\eta)}
\end{equation}
for $1/2+3\eta \leq \sigma \leq 3/2+\eta$. 

For a bound for $J_{-1/2-4\eta}(s)$ on the line $\sigma=\frac12+3\eta$, 
we note the relation 
$$
J_{-1/2-4\eta}\left(\frac12+3\eta+it\right)=J_{-1/2-2\eta}\left(\frac12+3\eta+it\right)+\frac12\zeta^k
\left(\frac12+3\eta+it\right) 
$$
implicit in Lemma~\ref{lemma3}: since $\zeta^k(1/2+3\eta+it) \ll (1+|t|)^{k/6} \ll 1+|t|^{2/3}$, 
it follows by this relation and \eqref{order2} that 
\begin{equation} \label{order3}
J_{-1/2-4\eta}\left(\frac12+3\eta+it\right) \ll |t|^{1-\frac23 \eta}.
\end{equation}

Lastly, we require a bound for $J_{-1/2-4\eta}(s)$ on the line $\sigma=-\frac12+5\eta$.  
Stirling's formula and a pointwise estimate for  $\zeta(s)$ (see \cite[(5.1.1)]{T}) give: 
\begin{align*}
\Gamma(-1+\eta+i(t+v)) & \ll (1+|t+v|)^{\eta-3/2}e^{-\frac{\pi}{2}|t+v|}, \\
\Gamma\left(\frac12+4\eta-iv\right) & \ll (1+|v|)^{4\eta} e^{-\frac{\pi}{2}|v|}, \\
\zeta(-1+\eta+i(t+v)) & \ll (1+|t+v|)^{\frac32-\eta}, \\[1ex]
\Gamma\left(-\frac12+5\eta+it \right) & \sim \sqrt{2\pi} t^{-1+5\eta}e^{-\frac{\pi}{2}t}
\end{align*}
for $t \geq 1$ and any real $v$. Hence,  
by an argument similar to that which gave us \eqref{order2}, we find that 
\begin{equation} \label{order4}
J_{-1/2-4\eta}\left(-\frac12+5\eta+it\right) \ll |t|^{2-\eta}. 
\end{equation}
By \eqref{order3}, \eqref{order4} and the convexity theorem, we find that one certainly has 
\begin{equation} \label{order5}
J_{-1/2-4\eta}\left(\sigma+it\right) \ll |t|^{-\sigma+3/2+4\eta}
\end{equation}
for $-1/2+5\eta \leq \sigma \leq 1/2+3\eta$ and $|t|\gg 1$.

\subsection{Estimation of $U_1$} \label{U1} 

As in the previous section let $\eta$ be any fixed small positive number. Let $P_j$ and $P_j'$ denote the points given by 
$P_0=3/2+\eta-iT, \, P_0'=3/2+\eta+iT$, $P_1=1/2+3\eta-iT, \, P_1'=1/2+3\eta+iT$, 
$P_2=-1/2+5\eta-iT, \, P_2'=-1/2+5\eta+iT$. 
Let $H_j$ and $H_j'$ denote the horizontal line segments given by 
$$
H_1: P_0 \to P_1, \ H_2: P_1 \to P_2  
$$
and
$$
H_1': P_1' \to P_0', \ H_2': P_2' \to P_1'. 
$$
Furthermore let $L_j$ denote the vertical line segments given by
$$
L_1: P_1 \to P_1', \ L_2:P_2 \to P_2'. 
$$
By Cauchy's theorem we have
\begin{equation} \label{U1-henkei-1}
U_1=\frac{1}{2\pi i}\left(\int_{H_1}+\int_{L_1}+\int_{H_1'}\right) F(s)\frac{x^s}{s}ds.  
\end{equation}
By Lemma \ref{lemma3}  the middle term on the right-hand side of \eqref{U1-henkei-1} becomes
\begin{align} \label{U1-henkei-2}
\frac{1}{2\pi i}\int_{L_1}F(s)\frac{x^s}{s}ds
&=\frac{1}{2\pi i}\int_{L_1}\left(F_0(s)+J_{-1/2-4\eta}(s)\right)\frac{x^s}{s}ds   \\
 & \quad {} -\frac{1}{2\pi i}\int_{L_1}\frac12\,\zeta^k(s)\frac{x^s}{s}ds. \notag
\end{align}
By Cauchy's theorem again the first term on the right-hand side of \eqref{U1-henkei-2} becomes
\begin{align} \label{U1-henkei-3}
&\frac{1}{2\pi i}\int_{L_1}\left(F_0(s)+J_{-1/2-4\eta}(s)\right)\frac{x^s}{s}ds  \\ 
& =\frac{1}{2\pi i}\left(\int_{H_2}+\int_{L_2}+\int_{H_2'}\right) \left(F_0(s)+J_{-1/2-4\eta}(s)\right)\frac{x^s}{s}ds +F_0(0). \notag
\end{align}
We should note that $J_{-1/2-4\eta}(0)=0$. Collecting \eqref{U1-henkei-1}--\eqref{U1-henkei-3} we get the expression
\begin{align*} 
U_1&=\frac{1}{2\pi i}\left(\int_{H_1}+\int_{H_1'}\right)F(s)\frac{x^s}{s}ds \nonumber \\
  &\quad  +\frac{1}{2\pi i}\left(\int_{H_2}+\int_{H_2'}\right)\left(F_0(s)+J_{-1/2-4\eta}(s)\right)\frac{x^s}{s}ds \nonumber \\
  &\quad +\frac{1}{2\pi i}  \int_{L_2} \left(F_0(s)+J_{-1/2-4\eta}(s)\right)\frac{x^s}{s}ds  \nonumber \\
  &\quad -\frac{1}{2\pi i} \int_{L_1}\frac12 \, \zeta^k(s)\frac{x^s}{s}ds +F_0(0). 
\end{align*}

We now estimate these integrals. For the integrals along the horizontal line segments 
$H_j, H_j'$, we see by \eqref{order-F-kukan} and \eqref{order5} that 
\begin{align*}
&\left(\int_{H_1}+\int_{H'_1}\right) F(s) \frac{x^s}{s} ds 
 \ll \frac{x^{3/2+\eta}}{T}+x^{1/2+3\eta}T^{-2\eta/3}, \\ 
&\left(\int_{H_2}+\int_{H'_2}\right) \left(F_0(s)+J_{-1/2-4\eta}(s)\right)\frac{x^s}{s}ds 
 \ll  x^{1/2+3\eta}T^{\eta}+ x^{-1/2+5\eta}T^{1-\eta}.
\end{align*}
Here we have used the bound $\zeta^{(j)}(\sigma-1+iT) \ll T^{3/2-\sigma+\eta}$ for 
$-1/2+5\eta \leq \sigma \leq 1/2+3\eta, \ 0 \leq j \leq 3$.
Using \eqref{order4} we get the bound
\begin{align*}
\int_{L_2}\left(F_0(s)+J_{-1/2-4\eta}(s)\right)\frac{x^s}{s}ds  \ll x^{-1/2+5\eta}T^{2-\eta}.
\end{align*}
As for the integral over $L_1$, we first recall (see Section \ref{sec5}) that when $k=3$ or $4$ 
one has $\int_0^U |\zeta(\sigma+it)|^kdt \ll U$   
for any fixed $\sigma\in(\frac12,1)$.   
It follows that 
\begin{align*} 
\left|\int_{L_1}\zeta^k(s)\frac{x^s}{s}ds\right|
 &\leq 2 x^{\frac12 + 3\eta} \int_0^T \left( {\textstyle\frac14} + t^2\right)^{-\frac12} 
\left|\zeta\left( {\textstyle\frac12}+3\eta +it\right)\right|^k dt \\
 &\ll x^{\frac12 + 3\eta}\left( 1 + \log_2(T) 
\cdot\max_{1\leq U\leq T} \frac{1}{U}\int_{U/2}^U 
\left|\zeta\left( {\textstyle\frac12}+3\eta +it\right)\right|^k dt\right)  \\ 
 &\ll x^{\frac12 + 3\eta} (1+\log T) \ll x^{\frac12 + 3\eta} T^{\eta} . 
\end{align*}

Combining these estimates and noting that $F_0(0) \ll 1$, and that 
$$
\frac{x^{\frac32+\eta}T^{-1}+x^{-\frac12+5\eta}T^{2-\eta}}{2} \geq 
\sqrt{x^{\frac32+\eta}T^{-1}\cdot x^{-\frac12+5\eta}T^{2-\eta}} 
\geq x^{\frac12+3\eta}T^{\eta}\geq 1 , 
$$
we get 
\begin{equation}  \label{U1-hyouka}
U_1 \ll \frac{x^{3/2+\eta}}{T}+x^{-1/2+5\eta}T^{2-\eta}.
\end{equation}

\subsection{Estimation of $U_2$}

By Perron's formula and \eqref{D34-bound} we get 
\begin{align}
U_2&=\sum_{n \leq x}d_k(n)+O\left(x^{\eta}\right)+O\left(\frac{x^{3/2+\eta}}{T}\right) \label{U2-hyouka} \\
  &=xP_k(\log x)+\Delta_k(x)+O\left(x^{\eta}\right)+O\left(\frac{x^{3/2+\eta}}{T}\right) \notag \\
  &=xP_k(\log x)+O\left(x^{1/2+\eta}\right)+O\left(\frac{x^{3/2+\eta}}{T}\right).  \notag
\end{align}

\subsection{Estimation of $U_3$} 

We move the line of integration to $L_1$. Noting that 
$$
\frac{\Gamma'(s-1)}{\Gamma(s-1)} \ll \log |t|, \ \ \frac{\Gamma''(s-1)}{\Gamma(s-1)} \ll \log^2 |t|,
\ \ \frac{\Gamma'''(s-1)}{\Gamma(s-1)} \ll \log^3 |t|
$$
(on the relevant horizontal line segments),  
and that $R_{-1}(s)$ has a pole of order $k$ at $s=1$, we get
\begin{align*}
U_3&=
\int_{1/2+3\eta-iT}^{1/2+3\eta+iT}R_{-1}(s)\frac{x^s}{s}ds \\ 
 &{\phantom{=}}+\Res_{s=1} \left(R_{-1}(s)\frac{x^s}{s}\right) 
+O\left(\frac{x^{3/2+\eta}T^{1/6}}{T^{2}}\right) 
+O\left(\frac{x^{1/2+3\eta}}{T^{1+2\eta}}\right).
\end{align*}
For the first term on the right-hand side of this equality we have the estimate 
$$
\int_{1/2+3\eta-iT}^{1/2+3\eta+iT}R_{-1}(s)\frac{x^s}{s}ds 
\ll x^{1/2+3\eta} \int_{-T}^T \frac{(\log(2+|t|))^{k-1}}{(1+|t|)^{1+2\eta}} dt  
\ll x^{1/2+3\eta}. 
$$
On the other hand, from \eqref{R3} and \eqref{R4}, a direct calculation shows 
\begin{multline*} 
\Res_{s=1} \left(R_{-1}(s)\frac{x^s}{s}\right) \\ 
=-\zeta(0)x \left(\frac12\log^2 x+\Bigl(3\gamma_0-1\Bigr)\log x 
+3\gamma_0^2-3\gamma_0+3\gamma_1+1 \right) 
\end{multline*} 
for $k=3$, and 
\begin{multline*}   
\Res_{s=1}\left(R_{-1}(s)\frac{x^s}{s}\right) \\ 
\begin{split} 
 &=-\zeta(0)x  
\biggl(\frac16 \log^3x 
+\Bigl(2\gamma_0-\frac12\Bigr)\log^2x 
+\Bigl(4\gamma_1+6\gamma_0^2-4\gamma_0+1\Bigr)\log x  \\ 
 &{\phantom{=}} +4\gamma_2-4\gamma_1+12\gamma_0\gamma_1+4\gamma_0^3-6\gamma_0^2+4\gamma_0-1\biggr)
\end{split} 
\end{multline*} 
for $k=4$. 
Comparing this with \eqref{shukou} and \eqref{shukou4}, we see that  
\begin{equation} \label{res-rel}
\Res_{s=1}\left(R_{-1}(s)\frac{x^s}{s}\right)=\frac12\Res_{s=1}\left(\zeta^k(s)\frac{x^s}{s}\right)=\frac12xP_k(\log x)
\end{equation}
for $k=3$ and $4$. So we get
\begin{equation} \label{U3-hyouka}
U_3 = \frac12x P_k(\log x) +O\left(\frac{x^{3/2+\eta}}{T^{11/6}}\right)+O\left(x^{1/2+3\eta}\right). 
\end{equation}

\begin{remark} \label{Rem1}
We note that the property \eqref{res-rel} is valid for all $k\geq 1$. In fact, for $r$ such that 
$0<r<\frac18$, $R_{-1}(s)$ has an expression
\begin{equation*}
R_{-1}(s)=\frac{1}{2\pi i}\int_{C(-1, r)}\frac{\Gamma(s+z)\Gamma(-z)}{\Gamma(s)}\zeta(s+z)\zeta^k(-z)dz,
\end{equation*}
where $C(w, r)$ denotes the circular path with positive direction of radius $r$ and the center $w$. The above 
expression gives an analytic continuation of $R_{-1}(s)$ into $\Im(s)> r$.  
Suppose that 
$|s-(1+2 r i)|<r$. In the annulus bounded by $C(-1,r)$ and $C(-1,3r)$, $\Gamma(s+z)\Gamma(-z)\zeta(s+z)\zeta^k(-z)$ as a 
function of $z$ has only one pole (which is at $z=-s$), therefore we have
\begin{align}
&R_{-1}(s)-\frac{1}{2\pi i} \int_{C(-1, 3r)}\frac{\Gamma(s+z)\Gamma(-z)}{\Gamma(s)}\zeta(s+z)\zeta^k(-z)dz \label{res-rel-1} \\
&=-\Res_{z=-s}\left(\frac{\Gamma(s+z)\Gamma(-z)}{\Gamma(s)}\zeta(s+z)\zeta^k(-z)\right) =\frac12 \zeta^k(s). \notag 
\end{align}
Since  the integral in \eqref{res-rel-1} is holomorphic in $|s-1|<3r$, \eqref{res-rel-1} gives an analytic continuation of 
$R_{-1}(s)$ into the punctured disk $0<|s-1|<3r$. From this observation and \eqref{Pk} we get
$$
\Res_{s=1}\left(R_{-1}(s)\frac{x^s}{s}\right)=\frac{1}{2} \Res_{s=1}\left(\zeta^k(s)\frac{x^s}{s}\right)=\frac{1}{2} xP_k(\log x)
$$
since $x^s/s$ is holomorphic in $|s-1|<3r$. This proves \eqref{res-rel}.
\end{remark}

\subsection{Estimation of $U_4$} 

Recall that 
\begin{equation} \label{U4sekibun}
U_4=\frac{1}{2\pi i} \int_{\kappa-iT}^{\kappa+iT}\frac{\zeta^k(s-1)}{s-1}\frac{x^s}{s}ds 
=\frac{1}{2\pi i} \int_{1/2+\eta-iT}^{1/2+\eta+iT}\zeta^k(s)\frac{x^{s+1}}{s(s+1)}ds. 
\end{equation}
In the last integral we shift the line of integration to $\sigma=-\delta <0$. 
Then, by Cauchy's theorem and pointwise bounds for $|\zeta(s)|$, we obtain 
\begin{align} \label{U4}
U_4&=\frac{1}{2\pi i} \int_{-\delta-iT}^{-\delta+iT}\zeta^k(s)\frac{x^{s+1}}{s(s+1)}ds +\left(-\frac12\right)^k x  \\
& \qquad  +O\left(\frac{x^{3/2+\eta}}{T^{2-\frac{k}{3}(\frac12-\eta)}}\right)
        +O\left(\frac{x^{1-\delta}}{T^{2-k(\frac12+\delta)}}\right). \notag
\end{align}
For the integral on the right-hand side of \eqref{U4} we have following lemma.
\begin{lemma}  \label{voronoitype}
Let $k \geq 3$ be a positive integer. Let $0<\delta \leq \frac16$. 
Suppose that either $k=3$ and $\delta=\frac16$, or 
else $k \geq 4$ and $\delta$ is small. 
Then for $x \geq 1$ and $N \in \mathbb{N}$, we have
\begin{multline*}
\frac{1}{2\pi i} \int_{-\delta-iT}^{-\delta+iT}\zeta^k(s)\frac{x^{s+1}}{s(s+1)}ds  \\ 
\begin{split} 
 &= 
 \frac{x^{\frac32-\frac{3}{2k}}}{2\pi^2 \sqrt{k} }\sum_{n \leq N}\frac{d_k(n)}{n^{\frac12+\frac{3}{2k}}}
\cos\left(2\pi k (nx)^{\frac{1}{k}}+\frac{\pi(k+3)}{4}\right) \\ 
&{\phantom{=}} +O\left(x^{3/2-2/k}N^{1/2-2/k+\delta}\right),
\end{split}  
\end{multline*} 
where $T=2\pi\left(x(N+\frac12)\right)^{1/k}$.
\end{lemma} 

The proof of Lemma \ref{voronoitype} will be given in Section 7.


\section{The proofs of Theorem \ref{newthm} and Corollary \ref{cor1}} 

\begin{proof}[Proof of Theorem \ref{newthm}] 
For $k=3$, $T=2\pi(x(N+1/2))^{1/3}$, it follows by Lemma~\ref{voronoitype}  
and the case $\delta = \frac16$ of \eqref{U4} that   
\begin{align*} 
U_4 &=\frac{x}{2\pi^2\sqrt{3}}\sum_{n \leq N} \frac{d_3(n)}{n}\cos\left(6\pi(nx)^{\frac13}+\frac{3\pi}{2}\right) \\  
 &{\phantom{=}} -\frac18x +O\left(\frac{x^{3/2+\eta}}{T^{3/2+\eta}}\right) +O\left(x^{5/6}\right).  
\end{align*} 
This, together with \eqref{Deltakwa}, \eqref{U1-hyouka}, \eqref{U2-hyouka} and \eqref{U3-hyouka}, gives us  
\begin{align*}
\sum_{n \leq x} \Delta_3(n) &=\frac12 x P_3(\log x) 
+\frac{x}{2\pi^2\sqrt{3}}\sum_{n \leq N} \frac{d_3(n)}{n}\cos\left(6\pi(nx)^{\frac13}+\frac{3\pi}{2}\right)-\frac18x  \nonumber \\
& \quad   {}+O(x^{5/6})+O\left(\frac{x^{3/2+\eta}}{T}\right)+O\left(x^{-1/2+5\eta}T^{2-\eta}\right). 
\end{align*}
Substituting here $T=2\pi(x(N+1/2))^{1/3}$, and simplifying a little, we get the result \eqref{thm2-3} of Theorem \ref{newthm}. 

In the case $k=4$, $T=2\pi(x(N+1/2))^{1/4}$,  Lemma \ref{voronoitype} and \eqref{U4} give 
\begin{equation*} 
U_4= \frac{x^{9/8}}{4\pi^2}\sum_{n \leq N} \frac{d_4(n)}{n^{7/8}}\cos\left(8\pi(nx)^{\frac14}+\frac{7\pi}{4}\right)
+ O(xN^{\delta}) + O\left(\frac{x^{3/2+\eta}}{T^{4(1+\eta)/3}}\right), 
\end{equation*} 
where $\delta$ is any small positive number. Similarly to the case $k=3$ we obtain
\begin{align*}
\sum_{n \leq x} \Delta_4(n) & = \frac12 x P_4(\log x)  + \frac{x^{9/8}}{4\pi^2}\sum_{n \leq N} 
\frac{d_4(n)}{n^{7/8}}\cos\left(8\pi(nx)^{\frac14}+\frac{7\pi}{4}\right) \\ 
& \quad + O\left(\frac{x^{3/2+\eta}}{T}\right)+O\left(x^{-1/2+5\eta}T^{2-\eta}\right)+O\left(xN^{\delta}\right).
\end{align*}
Substituting $T=2\pi(x(N+1/2))^{1/4}$, 
taking $\delta=\eta\in(0,\frac19)$  and simplifying a little, we get the result \eqref{thm2-4} of Theorem \ref{newthm}. 
\end{proof} 

\begin{proof}[Proof of Corollary \ref{cor1}]   
For the first part (where, in effect, $k=3$), we have the trivial upper bound $O(\log^3 N)$ 
for the absolute value of the sum over $n$ in  \eqref{thm2-3}.  
If we take $N=[x]$ in \eqref{thm2-3}, then the $O$-term there becomes $O(x^{5/6+5\eta})$,
so that by choosing $\eta<1/30$ we get
\begin{align*}
\sum_{n \leq x} \Delta_3(n) \ll x \log^3 x 
\end{align*}
(the polynomial $P_3(u)$ being of degree $2$). 
This proves \eqref{cor1-3}.

For the second part (where, in effect, $k=4$), we have the trivial upper bound $O(N^{1/8+\varepsilon})$ 
for the absolute value of the sum over $n$ in \eqref{thm2-4}.  
So by taking $N=[x^{1/3}]$ and $0<\eta< \varepsilon/5$ we obtain 
$$
\sum_{n \leq x} \Delta_4(n) 
\ll x^{9/8}N^{1/8+\varepsilon} + x^{5/4+5\eta}N^{-1/4} 
\ll x^{7/6+\varepsilon}.
$$
This proves \eqref{cor1-4}. 
\end{proof} 

\begin{remark}  \label{Rem-msq1}
Let $V_k(x,N)$  ($k=3, 4$) denote 
the second terms on the right-hand sides of 
\eqref{thm2-3} and 
\eqref{thm2-4}, respectively, 
so that we have:
\begin{align}
V_3(x,N)&=\frac{x}{2\pi^2\sqrt{3}}\sum_{n \leq N} \frac{d_3(n)}{n}\cos\left(6\pi(nx)^{\frac13}+\frac{3\pi}{2}\right),  \label{V3} \\
V_4(x,N)&=\frac{x^{9/8}}{4\pi^2}\sum_{n \leq N}\frac{d_4(n)}{n^{7/8}}\cos\left(8\pi(nx)^{\frac14}+\frac{7\pi}{4}\right). \label{V4}
\end{align}
Using standard methods we can prove easily that 
\begin{align}
\int_1^X V_3^2(x,N)dx&=\frac{1}{72\pi^4}\left(\sum_{n=1}^{\infty}\frac{d_3^2(n)}{n^2}\right)X^{3}
+O(X^3N^{\varepsilon - 1})+O(X^{8/3}), \label{msq-3} \\
\int_1^X V_4^2(x,N)dx&=\frac{1}{104\pi^4}\left(\sum_{n=1}^{\infty}\frac{d_4^2(n)}{n^{\frac74}}\right)X^{\frac{13}{4}}
+O\!\left(X^{\frac{13}{4}}N^{\varepsilon - \frac34}\right) +O(X^3N^{\varepsilon}). \label{msq-4}
\end{align}
\end{remark}


\section{Proof of Lemma \ref{voronoitype}}  
We follow the standard method of deriving the truncated Vorono\"{i}-type formula for the integral
$I=\frac{1}{2\pi i}\int_{-\delta-iT}^{-\delta+iT}\zeta^k(s)\frac{x^{s+1}}{s(s+1)}ds$.
See Titchmarsh \cite[Chapter XII]{T}. 
Using the functional equation of the Riemann zeta-function and expanding $\zeta(1-s)$  as a Dirichlet series we get
\begin{align}
I&=\frac{x}{\pi^k}\sum_{n=1}^{\infty}\frac{d_k(n)}{n}\frac{1}{2\pi i}\int_{-\delta-iT}^{-\delta+iT}
\left(\Gamma(1-s)\sin\frac{\pi s}{2}\right)^k \frac{(2^k \pi^k nx)^{s}}{s(s+1)}ds \nonumber \\
&=2^k x^2 \sum_{n=1}^{\infty}d_k(n)\frac{1}{2\pi i}\int_{1+\delta-iT}^{1+\delta+iT}
\left(\Gamma(s)\cos\frac{\pi s}{2}\right)^k \frac{(2^k \pi^k nx)^{-s}}{(1-s)(2-s)}ds. \nonumber 
\end{align} 
Let $N$ be a positive integer and let $T=2\pi\left(x\left(N+1/2\right)\right)^{1/k}.$
Let $C_1$ and $C_2$ denote the lines joining the points $-i\infty, \ -iT, \ 1+\delta-iT$ in this order respectively,
and let $C_3$ and $C_4$ denote the lines joining the points $1+\delta+iT, \ iT, \ i\infty$ in this order respectively,
and finally let $C_0$ denote the polygonal path joining the above six points by straight lines. 
We then have
\begin{align*}
I&=2^kx^2\sum_{n \leq N}d_k(n)\frac{1}{2\pi i}\left\{\int_{C_0}-\sum_{j=1}^4 \int_{C_j} \right\} \\[1ex]
&\quad +2^kx^2\sum_{ n >N}d_k(n)\frac{1}{2\pi i}\left\{\int_{1+\delta-iT}^{1+\delta-i}+
\int_{1+\delta-i}^{1+\delta+i}+\int_{1+\delta+i}^{1+\delta+iT}\right\} \\[1ex]
&=2^kx^2\sum_{n \leq N}d_k(n)\frac{1}{2\pi i}
\int_{C_0}\left(\Gamma(s)\cos\frac{\pi s}{2}\right)^k \frac{(2^k \pi^k nx)^{-s}}{(1-s)(2-s)}ds \\[1ex]
&\quad + J_1+J_2+\cdots + J_7,
\end{align*}
say. Since $J_1=\bar{J_4}, J_2=\bar{J_3}, J_5=\bar{J_7}$ we note that 
$$
\sum_{l=1}^7 J_l=J_6+2\Re(J_3+J_4+J_7).
$$

First we consider $J_6$. Since $|\int_{1+\delta-i}^{1+\delta+i}| \ll (nx)^{-(1+\delta)}$, we have
$$
J_6 \ll x^2 \sum_{n >N}d_k(n)(nx)^{-(1+\delta)} \ll x^{1-\delta}N^{-\delta} (\log N)^{k-1}.
$$
As for $J_7$, Stirling's formula gives
\begin{align*}
&\left(\Gamma(s)\cos\frac{\pi s}{2}\right)^k \frac{(2^k \pi^k nx)^{-s}}{(1-s)(2-s)} \\
&=-\left(\frac{\pi}{2}\right)^{k/2}e^{-\frac{\pi i k}{4}}(2^k\pi^knx)^{-\sigma}t^{k(\sigma-1/2)-2}
e^{iF(t)}\left(1+O\left(\frac1t\right)\right),
\end{align*}
where 
$$
F(t)=k\left(t\log\frac{t}{2\pi}-t-t\log((nx)^{1/k})\right).
$$
We have here $k(\frac12 + \delta) - 2\geq 0$ and 
$$
F'(t)=-k\log\frac{2\pi(nx)^{1/k}}{t} \leq -k \log\frac{2\pi(nx)^{1/k}}{T} =-\log \frac{n}{N+1/2} , 
$$
and so, by the first derivative test \cite[Lemma~4.3]{T}, we find that for $n\geq N+1$ one has: 
\begin{align*}
&\int_{1+\delta+i}^{1+\delta+iT} \left(\Gamma(s)\cos\frac{\pi s}{2}\right)^k \frac{(2^k \pi^k nx)^{-s}}{(1-s)(2-s)}ds \\
& \hspace{2cm}   \ll (nx)^{-(1+\delta)}\left(\frac{T^{k(1/2+\delta)-2}}{\log \frac{n}{N+1/2}} +O(r_k(T))\right),
\end{align*}
where, by the choice of $\delta$, $r_3(T)=\log T$ and $r_k(T)=T^{k(1/2+\delta)-2}$ for $k \geq 4$.
Hence
$$
J_7 \ll x^{1-\delta}T^{k(1/2+\delta)-2}\sum_{n>N}\frac{d_k(n)}{n^{1+\delta}}\frac{1}{\log\frac{n}{N+1/2}}
+x^{1-\delta}N^{-\delta}(1+\log N)^{k-1}\, r_k(T).
$$
Therefore, noting that $d_k(n) \ll n^{\varepsilon}$, and that $\log\frac{n}{N+1/2} \gg 1$ for $n >2N$, while 
$\log \frac{n}{N+1/2} \gg \frac{n-N}{N}$ for $N+1 \leq n \leq 2N$, we find that
\begin{align*}
J_7 & \ll x^{1-\delta}T^{k(1/2+\delta)-2}  
      \ll x^{3/2-2/k}N^{1/2-2/k+\delta}. 
\end{align*}
Regarding $J_3$, we observe that 
\begin{align*}
\left|\int_{C_3}\left(\Gamma(s)\cos\frac{\pi s}{2}\right)^k \frac{(2^k \pi^k nx)^{-s}}{(1-s)(2-s)}ds\right| 
& \ll T^{-(k/2+2)}\int_0^{1+\delta}\left( 
\frac{T^k}{2^k \pi^k nx}\right)^{\sigma} d\sigma \\
&= T^{-(k/2+2)}\int_0^{1+\delta}\left(\frac{N+1/2}{n}\right)^{\sigma} d\sigma \\
& \ll  T^{-(k/2+2)}\left(\frac{N}{n}\right)^{1+\delta} 
\end{align*} 
when $n\leq N$. 
Hence 
\begin{align*}
J_3 & \ll x^2 N^{1+\delta}T^{-(k/2+2)}\sum_{n \leq N}\frac{d_k(n)}{n^{1+\delta}} \\
    &  \ll x^{3/2-2/k}N^{1/2-2/k+\delta}.
\end{align*}
As for $J_4$, we note it follows by the first derivative test that 
$$
\int_{iT}^{i\infty} \left(\Gamma(s)\cos\frac{\pi s}{2}\right)^k \frac{(2^k \pi^k nx)^{-s}}{(1-s)(2-s)}ds 
\ll T^{-k/2-2}\left(\frac{1}{\log \frac{N+1/2}{n}}+1 \right)  
$$
when $n\leq N$. Hence 
\begin{align*}
J_4 & \ll x^2 T^{-k/2-2} \sum_{n \leq N}d_k(n)\left(\frac{1}{\log\frac{N+1/2}{n}}+1\right) \\
& \ll x^2 T^{-k/2-2}N^{1+\delta} \\[1ex]
& \ll x^{3/2-2/k}N^{1/2-2/k+\delta}.
\end{align*}

Collecting the above estimates, we obtain  
\begin{align} \label{trV-1}
I&=2^kx^2\sum_{n \leq N}d_k(n)\frac{1}{2\pi i}
\int_{C_0}\left(\Gamma(s)\cos\frac{\pi s}{2}\right)^k \frac{(2^k \pi^k nx)^{-s}}{(1-s)(2-s)}ds  \\[1ex]
&\quad +O\left(x^{3/2-2/k}N^{1/2-2/k+\delta}\right) \nonumber
\end{align}
(note that we take $\delta=1/6$ in the case $k=3$).

Now we consider the first term on the right-hand side of \eqref{trV-1}. 
The treatment here is due to Tong \cite{To}. See also \cite{CTZ}. 
Let $K(y)$ be the function defined by  
\begin{align*}
K(y)&=\frac{1}{2\pi i}\int_{C_0}\left(\Gamma(s)\cos\frac{\pi s}{2}\right)^k \frac{y^{-ks}}{(1-s)(2-s)}ds, \ (y>0).
\end{align*}
To find the asymptotic behaviour of $K(y)$, as $y\rightarrow +\infty$, we transform the path $C_0$ 
so as to have it pass  through the so-called ``saddle point".  
Let $s$ be a point in the upper half plane and introduce new parameters $w$, $\xi$ by
\begin{align*}
s=yw, \ \ w=i(1+\xi). 
\end{align*}
Then, for $t>0$, Stirling's formula gives: 
\begin{align}
&\left(\Gamma(s)\cos\frac{\pi s}{2}\right)^k \frac{y^{-ks}}{(1-s)(2-s)} \label{stirling-w}  \\
&=\left(\frac{\pi}{2}\right)^{k/2} \exp\left(k\Bigl(\bigl(s-\frac12\bigr)\log s - s-\frac{\pi i s}{2}-s\log y\Bigr)\right)
\frac{1}{s^2}\left(1+O\left(\frac{1}{|s|}\right)\right) \notag \\[1ex]
&=\left(\frac{\pi}{2}\right)^{k/2} \exp\left(k\Bigl(yw\big(\log w -1-\frac{\pi i}{2}\bigr)-\frac12\log s\Bigr)\right)
\frac{1}{s^2}\left(1+O\left(\frac{1}{|s|}\right)\right) \notag \\[1ex]
&=\left(\frac{\pi}{2}\right)^{k/2} \exp\Bigl(kyi\bigl((1+\xi)\log(1+\xi)-\xi\bigr)\Bigr)
\frac{e^{-kyi}}{s^{k/2+2}}\left(1+O\left(\frac{1}{|s|}\right)\right). \notag 
\end{align}
Let $u$ be the function of $\xi$ defined by
$$
u=u(\xi)=\Bigl(-ki\bigl((1+\xi)\log(1+\xi)-\xi\bigr)\Bigr)^{1/2}=\sqrt{\frac{k}{2}}e^{-\frac{\pi i}{4}}
\xi \sqrt{\frac{(1+\xi)\log(1+\xi)-\xi}{\xi^2/2}},
$$
where $\sqrt{z}$ denotes the principal branch of the function $z^{1/2}=\exp(\frac12 \log z)$.
Then $u$ is analytic in $|\xi|<1$ and $u'(0)\neq 0$. Thus by inversion we have 
\begin{equation} \label{xitou}
\xi(u)=\sqrt{\frac2k}\, e^{\frac{\pi i}{4}}u+\frac{i}{3k}u^2+O(u^3)
\end{equation}
in a neighbourhood of $u=0$. For a small constant $r>0$ we put
$$
w_1=u_1+iv_1=i(1+\xi(-r)), \ \ w_2=u_2+iv_2=i(1+\xi(r)).
$$
Let ${\cal C}_j\ (j=0,1,2)$ be the paths defined by 
\begin{align*}
{\cal C}_0&=\{s \mid s=yi(1+\xi(u)), \ -r \leq u \leq r\},  \\
{\cal C}_1&=\{s \mid s=y(u_1+iv), \ 0 \leq v \leq v_1\}, \\
{\cal C}_2&=\{s \mid s=y(u_2+iv), \ v \geq v_2\}
\end{align*}
and let ${\cal C}_j '$ be the complex conjugate of ${\cal C}_j$ with the reverse direction. 
Let 
$$
K_j^{+}(y)=\frac{1}{2\pi i}\int_{{\cal C}_j} \left(\Gamma(s)\cos\frac{\pi s}{2}\right)^k 
\frac{y^{-ks}}{(1-s)(2-s)}ds , 
$$
and let $K_j^{-}(y)$ be the corresponding integral over ${\cal C}_j'$. 
Since $K_j^{-}(y)=\overline{K_j^{+}(y)}$, we have, by Cauchy's theorem, 
\begin{equation} \label{Ky}
K(y)= 2 \, \Re \left(\sum_{j=0}^2 K_j^{+}(y) \right)+(-1)^{k+1} y^{-2k}
\end{equation}
for large  $y$.  

The main term in the asymptotic formula for $K(y)$ comes from the integral $K_0^{+}(y)$, in which we have 
\begin{equation} \label{ds-to-du} 
ds=\sqrt{\frac{2}{k}}e^{\frac{\pi i}{4}}yi \left(1+\frac13\sqrt{\frac{2}{k}}
e^{\frac{\pi i}{4}}u+O(u^2)\right)du  
\end{equation} 
(with $u$ running over the interval $[-r,r]$).  
Specifically, it follows from \eqref{stirling-w}, \eqref{xitou} and \eqref{ds-to-du} that 
\begin{align*}
K_0^{+}(y)&=\frac{1}{2\pi}\left(\frac{\pi}{2}\right)^{\frac{k}{2}}\sqrt{\frac{2}{k}}
e^{-i\left(ky+\frac{\pi}{4}(k+3)\right)}y^{-\frac{k}{2}-1}\int_{-r}^{r}\exp(-yu^2)\\
& \quad \times \left(1-\Bigl(\frac{k}{2}+\frac53\Bigr)\sqrt{\frac{2}{k}}e^{\frac{\pi i}{4}}u+O(u^2)\right)
\left(1+O\Bigl(\frac{1}{y}\Bigr)\right)du , 
\end{align*}
and so, observing that  
the above integral equals $y^{-1/2}\bigl(\sqrt{\pi}+O(e^{-yr^2/2})\bigr) +O(y^{-3/2})$, we get 
\begin{align*}
K_0^{+}(y)&=\frac{1}{2\sqrt{k}}\left(\frac{\pi}{2}\right)^{\frac{k}{2}-\frac12}
e^{-i\left(ky+\frac{\pi}{4}(k+3)\right)}y^{-\frac{k}{2}-\frac{3}{2}}+O(y^{-\frac{k}{2}-\frac{5}{2}}). 
\end{align*}

Now we estimate $K_2^{+}(y)$.  Writing $s=yw, \ w=u_2+iv$ in \eqref{stirling-w}, we have
\begin{equation} \label{K2plushyouka} 
K_2^{+}(y) \ll y^{-\frac{k}{2}-1} \int_{v_2}^{\infty} \frac{dv}{|w|^{\frac{k}{2} + 2}}  
\cdot\max_{v \geq v_2} \left| \exp\left(kyw \left(\log w-1-\frac{\pi i}{2}\right)\right) \right| . 
\end{equation}
Since
$$
\frac{d}{dv}\Re\left(w\left(\log w-1-\frac{\pi i}{2}\right)\right)=-\tan^{-1}\frac{-u_2}{v}<0,
$$
the real part of $kyw\left(\log w-1-\frac{\pi i}{2}\right)$ is a decreasing function of $v$ (for $v>0$), 
and so the maximum that appears in \eqref{K2plushyouka} is attained at $v=v_2$, 
where one has $w=u_2+iv_2=w_2=i(1+\xi(r))$ and, therefore, 
\begin{align*}
\Re\left( w\left(\log w-1-\frac{\pi i}{2}\right)\right)
&=-\Im \left(\Bigl(1+\xi(r)\Bigr)\Bigl(\log\bigl(1+\xi(r)\bigr)-1 \Bigr) \right) \\
&=-\Im \Bigl((1+\xi(r))\log\bigl(1+\xi(r)\bigr)-\xi(r)\Bigr) \\
&=-r^2/k.
\end{align*}
We also note that the integral on the right-hand side of \eqref{K2plushyouka} is less than $\int_1^{\infty}v^{-k/2-2}dv$
if $r$ is sufficiently small. Hence we get
$$
K_2^{+}(y) \ll y^{-k/2-1}e^{-r^2 y}. 
$$
It can be shown similarly that 
$$
K_1^{+}(y) \ll y^{-k/2-1}e^{-r^2y}. 
$$

Combining these estimates we get
$$
\sum_{j=0}^2 K_j^{+}(y)=\frac{1}{2\sqrt{k}}\left(\frac{\pi}{2}\right)^{\frac{k}{2}-\frac12}
e^{-i\left(ky+\frac{\pi}{4}(k+3)\right)}y^{-\frac{k}{2}-\frac{3}{2}}+O(y^{-\frac{k}{2}-\frac{5}{2}}), 
$$
and hence, by \eqref{Ky}, 
\begin{equation} \label{Ky-asymp}
K(y)=\frac{1}{\sqrt{k}}\left(\frac{\pi}{2}\right)^{\frac{k}{2}-\frac12}y^{-\frac{k}{2}-\frac32}
\cos\left(ky+\frac{\pi}{4}(k+3)\right)+O\left(y^{-\frac{k}{2}-\frac52}\right).
\end{equation}

From \eqref{trV-1} and \eqref{Ky-asymp} we get
\begin{align*}
I&=2^kx^2\sum_{n \leq N}d_k(n)K(2\pi(nx)^{1/k}) 
+O\left(x^{3/2-2/k}N^{1/2-2/k+\delta}\right) \\
&=\frac{x^{\frac32-\frac{3}{2k}}}{2\pi^2 \sqrt{k} }\sum_{n \leq N}\frac{d_k(n)}{n^{\frac12+\frac{3}{2k}}}
\cos\left(2\pi k (nx)^{\frac{1}{k}}+\frac{\pi(k+3)}{4}\right)   \\
& \quad + O\left(x^{\frac{3}{2}-\frac{5}{2k}} \sum_{n \leq N}\frac{d_k(n)}{n^{\frac12+\frac{5}{2k}}}\right) 
  +O\left(x^{3/2-2/k}N^{1/2-2/k+\delta}\right).
\end{align*}
By hypothesis, we have 
$k\geq 3$ (and $\delta = \frac16$ if $k=3$; $\delta >0$ otherwise).  Hence the last of the above $O$-terms  
absorbs the second last,  so that the stated  result of Lemma \ref{voronoitype} is obtained. \qed

\begin{remark}
For an asymptotic expansion of the integral $K(y)$, see also Chandrasekharan and Narashimhan \cite{CN} and Hafner \cite{Haf}.
\end{remark}


\section{Proof of Theorem \ref{thm3}}

We deduced the results \eqref{cor1-3} and \eqref{cor1-4} of Corollary~\ref{cor1} 
using trivial estimates for the sums over $n$ on the right-hand side of \eqref{thm2-3} and \eqref{thm2-4}. 
By a closer study of such sums (employing estimates from the theory of exponential sums), 
we can improve upon \eqref{cor1-3} and \eqref{cor1-4}. 

First we consider the case $k=3$. Let 
\begin{align} 
S_3&=\sum_{n \leq N} \frac{d_3(n)}{n}\cos\left(6\pi(nx)^{\frac13}+\frac{3\pi}{2}\right) \label{Def-S3}\\ 
&=\sum_{n_1n_2n_3 \leq N} \frac{1}{n_1n_2n_3}\cos\left(6\pi(xn_1n_2n_3)^{\frac13}+\frac{3\pi}{2}\right), \notag 
\end{align} 
where 
$2\leq N\leq x$ (we may assume that $x\geq 2$). 
By symmetry, the part of the sum $S_3$ with $n_3 \leq n_2 \leq n_1$ equals to $\frac16 S_3+O(\log N)$.
One can deduce from this that
\begin{equation} \label{S3}
S_3 \ll \sum_{n_1 \leq N}\sum_{n_2 \leq \min\left(n_1, \frac{N}{n_1}\right)}\frac{1}{n_1n_2}
\left|\sum_{n _3 \leq \min\left(n_2, \frac{N}{n_1n_2}\right)}\frac{e\left(3(xn_1n_2n_3)^{\frac13}\right)}{n_3}\right|,
\end{equation}
where $e(u)=e^{2\pi iu}$. Note that the sum on the right-hand side of \eqref{S3} is certainly greater than or equal to
$\sum_{n_1 \leq N}\frac{1}{n_1} \gg \log N$. 

We put $N'=\min\left(n_2, \frac{N}{n_1n_2}\right)$ for short.
To estimate the sum over $n_3$ on the right-hand side of \eqref{S3}, define 
$$
S(M)=\sum_{M < n_3 \leq 2M} e(f(n_3))
$$
for $M \leq N'/2$, where $f(u)=3(xn_1n_2)^{\frac13} u^{\frac13}$. 
Since $f'(u) =(xn_1n_2)^{\frac13} u^{-\frac23}$, we see that $\frac{F}{2^{2/3}M} \leq f'(u) \leq \frac{F}{M}$ 
with $F=F(M):=(xn_1n_2)^{\frac13}M^{\frac13}$. 
Noting that $F/M =(xn_1n_2)^{\frac13}M^{-\frac23} \geq x^{\frac13}(n_1/n_2)^{\frac13} \geq 
x^{\frac13} > 1$, we have 
\begin{equation*} 
S(M) \ll M^{\lambda}\left(\frac{F}{M}\right)^{\kappa}
\end{equation*}
for any exponent pair $(\kappa, \lambda)\in [0,\frac12]\times[\frac12,1]$. 
If $M^4 > F$, 
then by taking $(\kappa, \lambda)=A^3BA^2B(0,1)=(\frac{1}{42}, \frac{25}{28})$  we get
\begin{equation} \label{sisuuwa-1}
S(M) \ll M^{\lambda+3\kappa}\ll M^{\frac{27}{28}}. 
\end{equation}

Suppose that $M^4 \leq F$. In this case  we apply the following result. 
\begin{lemma} \label{IK}
Let $M\geq 2$, and let $f$ be a real function satisfying 
\begin{equation} \label{lem5-1} 
\alpha^{-j^3}F \leq \frac{x^j}{j!}|f^{(j)}(x)| \leq \alpha^{j^3}F 
\end{equation} 
on $[M,2M]$ for all positive integers $j$, with $F \geq M^4$ and $\alpha \geq 1$. 
Then for all $a,b\in[M,2M]$ we have 
$$
\sum_{a<n \leq b} e(f(n)) \ll \alpha M \exp\left(-2^{-18}\frac{(\log M)^3}{(\log F)^2}\right),
$$ 
where the implicit constant is absolute.
\end{lemma}
This is Iwaniec and Kowalski's \cite[Theorem 8.25]{IK}. 
Its proof is an application of I.M. Vinogradov's mean value theorem.   

\medskip

We apply Lemma \ref{IK} with $f(u)=3(xn_1n_2)^{\frac13}u^{\frac13}$. 
The conditions \eqref{lem5-1} are satisfied with $F=F(M)=(xn_1n_2)^{\frac13}M^{\frac13}$ (as before) 
and $\alpha=3/2$. Hence 
\begin{equation} \label{sisuuwa-2}
S(M) \ll M \exp\left(-2^{-18}\frac{(\log M)^3}{(\log F)^2}\right)
\end{equation} 
when $F\geq M^4$ and $M\geq 2$. Comparing \eqref{sisuuwa-1} and \eqref{sisuuwa-2}, 
one sees that the estimate \eqref{sisuuwa-2} holds regardless of whether or not the condition 
$F\geq M^4$ is satisfied: for the right-hand side of \eqref{sisuuwa-2} is greater than 
$M \exp\left(-2^{-18}\log M\right)=M^{1-\frac{1}{2^{18}}} \geq M^{\frac{27}{28}}$, provided only that $F>M\geq 2$. 

Let $1\leq M_0\leq N'$. Then
\begin{align*}
\sum_{1 \leq n_3 \leq M_0} e(f(n_3)) = \sum_{j=1}^{\infty}S(M_j)
&= \sum_{1\leq j\leq J}S(M_j) +\sum_{j>J}S(M_j)\\ 
&=:S_1+S_2\qquad\text{(say)} , 
\end{align*}
where $J=J(M_0):=\log_2(\sqrt{M_0})\geq 0$ and $M_j:=2^{-j}M_0$. 
Put $F_j=F(M_j)= (xn_1n_2)^{\frac13}M_j^{\frac13}$ for $j\geq 0$. 
When $j\geq 1$ we have both $\log F_j \leq \log F_1 <\frac13\log(xn_1n_2M_0) 
\leq \frac13 \log(xn_1n_2N')\leq\frac13 \log(xN)<\log x$ 
and (see our earlier discussion) $F_j / M_j = F(M_j) / M_j > F(N') / N' \geq x^{\frac13}>1$.  
In cases where $1\leq j\leq J$ we have also $\log(M_j) \geq \frac12 \log M_0\geq \log 2$ 
(note that $J\geq 1$ only if $M_0\geq 4$). Thus 
\begin{equation*} 
S(M_j) \ll M_j \exp\left(-2^{-18} \frac{(\log M_j)^3}{(\log F_j)^2}\right) 
< 2^{-j}M_0 \exp\left(-2^{-21} \frac{(\log M_0)^3}{(\log x)^2}\right) 
\end{equation*} 
when $1\leq j\leq J$. Therefore $ S_1 \ll M_0\exp\left(-2^{-21} (\log M_0)^3 / (\log x)^2\right)$. 
This, when combined with the trivial bound $|S_2|<2M_J=\sqrt{4M_0}$, gives us 
\begin{equation} \label{sisuuwa-3}
\sum_{1 \leq n_3 \leq M_0} e(f(n_3))\ll M_0 \exp\left(-\beta\frac{(\log M_0)^3}{(\log x)^2}\right) 
\end{equation}
where the implicit constant is absolute and $\beta =2^{-21}$.  
In particular, since $x^2\geq 2x > N\geq (N')^2 \geq M_0^2$, we have $x > M_0\geq 1$, so that 
the right-hand side of \eqref{sisuuwa-3} is greater than or equal to $M_0^{1-\beta}\geq\sqrt{M_0}>|S_2|/2$. 

We have shown that \eqref{sisuuwa-3} holds whenever $M_0 \in [1,N']$. We deduce, by partial summation, that  
\begin{align}
\sum_{1 \leq n_3 \leq N'}\frac{e(f(n_3))}{n_3} &\ll 
\exp\left(-\beta\frac{(\log N')^3}{(\log x)^2}\right) +\int_1^{N'} \frac{1}{u} 
\exp\left(-\beta\frac{(\log u)^3}{(\log x)^2}\right) du \label{sisuuwa-4}  \\
& \leq 1+ \int_0^{\infty} \exp\left(-\beta\frac{v^3}{(\log x)^2}\right) dv \notag \\
& \ll (\log x)^{\frac23}. \notag 
\end{align}
It follows from \eqref{Def-S3}, \eqref{S3} and \eqref{sisuuwa-4} that we have
$$
\sum_{n \leq N} \frac{d_3(n)}{n}\cos\left(6\pi(nx)^{\frac13}+\frac{3\pi}{2}\right) 
\ll (\log x)^{\frac23}(\log N)^2 
$$ 
for $x\geq N\geq 2$. Hence, by choosing $N=[x]$ and $\eta = \frac{1}{60}$ (say) in \eqref{thm2-3}, 
we obtain the bound  
$$
\sum_{n \leq x} \Delta_3(n) \ll x(\log x)^{\frac83}, 
$$
which is \eqref{thm3-3}, the first part of Theorem \ref{thm3}. 
\par 
Next we treat the case $k=4$, assuming (as we may) that $x\geq 3$. 
For $x^{1/3} < N \leq x^{1/2}$  we consider the sum 
\begin{equation}\label{s4}
S_4=\sum_{n \leq N} \frac{d_4(n)}{n^{7/8}}\cos\left(8\pi(nx)^{\frac14}+\frac{7\pi}{4}\right). 
\end{equation}
By a  splitting argument, $S_4$ is no greater in modulus than a sum of $O\left((\log x)^4\right)$ sums of the form 
\begin{multline} \label{s42-1} 
\sum_{\substack{m_1m_2m_3m_4 \leq N \\ M_j<m_j\leq 2M_j\ (j=1,\ldots,4)}}
\frac{1}{(m_1m_2m_3m_4)^{7/8}}\, e\left(4(xm_1m_2m_3m_4)^{\frac14}\right) \\  
=: S^{\ast}(M_1,M_2,M_3,M_4) 
\end{multline} 
with $\{ 2+\log_2 M_j : 1\leq j\leq 4\}\subset {\mathbb N}$ and 
$N_0:=\prod_{j=1}^4 M_j < N\leq x^{1/2}$.
Therefore, we have 
\begin{equation}\label{VVV}
S_4 \ll  \left| S^{\ast}(M_1,M_2,M_3,M_4)\right| (\log x)^4 
\ll  \left| S^{\ast}(M_1,M_2,M_3,M_4)\right| x^{\varepsilon}\;,
\end{equation} 
where $(M_1,\ldots,M_4)$ is some $4$-tuple satisfying the conditions just mentioned. 
We may (by symmetry) assume here that $M_4\geq M_1,M_2,M_3\geq 1/2$, 
so that $N_0^{1/4}\leq M_4\leq 8N_0$.  
By positivity and partial summation, we have
\begin{multline} \label{s42-2} 
\left| S^{\ast}(M_1,M_2,M_3,M_4) \right|  \\ 
\quad\leq \frac{1}{(N_0 / M_4)^{7/8}} 
\sum_{\substack{(m_1,m_2,m_3) \\ M_j<m_j\leq 2M_j\ (j=1,2,3)}}  
\max_{M_4 < c \leq d \leq 2M_4}\left|\sum_{c \leq m_4 \leq d} \frac{1}{m_4^{\frac78}} \, e(f(m_4))\right| \\ 
\quad\ll N_0^{\frac18} M_4^{-1}  
\max_{\substack{(m_1,m_2,m_3,a,b) \\ M_j<m_j\leq 2M_j\ (j=1,2,3) \\ M_4 < a \leq b \leq 2M_4}}   
\left|\sum_{a \leq m_4 \leq b} e(f(m_4))\right| , 
\end{multline} 
where $f(u)= f_{x,m_1,m_2,m_3}(u) := 4(xm_1m_2m_3)^{1/4}u^{1/4}$. 
  
Let $(\kappa,\lambda)$ be an arbitrary exponent pair (to be specified later).
For $M_4\leq u\leq 2M_4$ and all $m_1,m_2,m_3$ satisfying the conditions in \eqref{s42-2}, we have  
$$f'(u)=(xm_1m_2m_3)^{\frac14}u^{-\frac34} \asymp V\;,$$ 
where $V=x^{1/4}N_0^{1/4}/M_4 
>N_0^{3/4}/M_4\gg M_4^{-1/4}\gg M_4^{-\lambda/(\kappa +1)}$,  
and so have 
\begin{align} \label{s42-3}
\sum_{a \leq m_4 \leq b} e(f(m_4)) 
\ll V^{\kappa} M_4^{\lambda} + V^{-1} 
\ll V^{\kappa} M_4^{\lambda} = x^{\frac{\kappa}{4}}N_0^{\frac{\kappa}{4}}M_4^{\lambda-\kappa}
\end{align}
when $M_4 < a\leq b\leq 2M_4$.  
We deduce from \eqref{s42-2} and \eqref{s42-3} that 
\begin{equation}\label{Type-I} 
S^{\ast}(M_1,M_2,M_3,M_4) \ll x^{\frac{\kappa}{4}}N_0^{\frac18+\frac{\kappa}{4}}M_4^{\lambda-\kappa-1}\;. 
\end{equation} 
Since $\lambda -\kappa - 1\leq \lambda - 1\leq 0$, 
it follows from \eqref{Type-I} that 
\begin{equation}\label{Type-I-1stUse} 
S^{\ast}(M_1,M_2,M_3,M_4) \ll x^{\frac{\kappa}{4}}N_0^{\frac{\lambda}{2}-\frac{\kappa}{4}-\frac38} 
\quad\text{if $\,M_4\geq \sqrt{N_0}$} .
\end{equation}
\par 
Let 
\begin{equation}\label{UltimateEP} 
(p,q) = \left(\varrho + \varepsilon , \varrho + {\textstyle\frac12} + \varepsilon\right) ,
\end{equation} 
where $\varrho$ is the constant defined in \eqref{Def-rho-constant} and $\varepsilon$ denotes, 
as usual, an arbitrarily small positive constant.  
By \eqref{Def-rho-constant}, \eqref{BourgainEP} and (see \cite[p. 55]{GK}) the convexity 
of the set $\mathscr{G}\subset{\mathbb R}^2$,  
we have $\frac12 - \varrho \geq \frac{29}{84} >0$, and 
$(p,q)$ is an exponent pair if $0<\varepsilon \leq \frac12 - \varrho$ 
(which we may henceforth assume is the case). 
We observe now that $(p,q)$ is certainly a point on the line segment with endpoints $(0,\frac12)$ 
and $(\frac12,1)$, and that the $A$ process of exponent pair theory 
maps exponent pairs that lie on this line segment 
to exponent pairs on the line segment with endpoints $(0,\frac34)$ and $(\frac16,\frac56)$    
(one has 
$\mathscr{G}\ni A(\kappa,\lambda) := (\frac{\kappa}{2\kappa + 2} , \frac12 + \frac{\lambda}{2\kappa +2})$ 
for $(\kappa ,\lambda)\in\mathscr{G}$). 
Therefore, putting now 
\begin{equation}\label{Def-KL-pair}
(P,Q) = A(p,q) = \left( \frac{p}{2p+2} , \frac{p+q+1}{2p+2}\right) 
\end{equation} 
(so that $(P,Q)$ is an exponent pair), we have  
\begin{equation}\label{KL-properties} 
2Q - P = {\textstyle\frac32}\quad\text{and}\quad P+Q \leq 1\;.
\end{equation}
In view of the first part of \eqref{KL-properties}, 
it follows by the special case $(\kappa,\lambda) = (P,Q)$ of  
\eqref{Type-I-1stUse} that 
\begin{equation}\label{Type-I-M_4-large}
S^{\ast}(M_1,M_2,M_3,M_4) \ll x^{\frac{P}{4}} \quad\text{if $\,M_4\geq \sqrt{N_0}$} .
\end{equation} 

With \eqref{Type-I-M_4-large}, our treatment of the case  
where one has $M_4\geq \sqrt{N_0}$ concludes.  
In treating the complementary case, where one has $M_4 < \sqrt{N_0}$,   
we shall use the following result, which is essentially a special case of Baker's \cite[Theorem~2]{B},  
and is also quite similar to \cite[Lemma~4.13]{GK} and the earlier result \cite[Lemma~III]{P-S} 
of Piatetski-Shapiro. 

\begin{lemma}\label{P-S_lemma} 
Let $\gamma$ and $\delta$ be non-zero real constants with $\min\{\gamma ,\delta\} < 1$, 
let $(k,l)$ be a fixed exponent pair, and let 
$(K,\Lambda)=(\frac{k}{2k+2},\frac{k+l+1}{2k+2})=:A(k,l)$. 
Then, for $\rho\in (0,\infty)$, $X,Y\in [1/2,\infty)$, $Z=\rho X^{\gamma} Y^{\delta}$, and all 
complex sequences $\{ b_n\}$ satisfying $|b_n|\leq 1$ ($Y<n\leq 2Y$), 
one has 
\begin{multline*} 
\left( 
\min\left\{ B_{\gamma}(X,Y,Z) , B_{\delta}(Y,X,Z)\right\} 
+\frac{XY}{\sqrt{\min\{ X,Y,Z\}}}\right) 
\log^2 (8XY) \\ 
\gg  \sum_{X<m\leq 2X} \max_{Y\leq u,v\leq 2Y} 
\left| \sum_{u<n\leq v} b_n e\left( \rho m^{\gamma} n^{\delta}\right)\right| 
=: E\;,
\end{multline*} 
say, where $B_{\alpha}(U,V,Z) := Z^K U^{\Lambda} V^{1-K} + \max\{ 0, [\alpha]\}\cdot UV$.
\end{lemma}

\begin{proof}
We may assume that $X$, $Y$ and $Z$ are large, since otherwise what the lemma 
asserts is trivially true. We assume, in particular, that $Z > 1$. 
By an application of \cite[Lemma~7.3]{GK}, one finds that 
$$ 
E\ll 
\sum_{X<m\leq 2X} 
\left| \sum_{Y<n\leq 2Y} b_n e\left( \rho m^{\gamma} n^{\delta} + n\theta\right)\right| 
\log(2+Y)
$$
for some $\theta$ that is independent of $m$ and $n$ and satisfies $0\leq \theta <1$. 
Thus, for certain complex numbers $\alpha_m$ and $\beta_n = b_n e(n\theta)$ 
whose modulus is less than or equal to $1$, one has: 
$$ 
\frac{E}{\log(2+Y)}\ll  
\sum_{X<m\leq 2X} \alpha_m 
\sum_{Y<n\leq 2Y} \beta_n e\left( \rho m^{\gamma} n^{\delta}\right) 
= E'\;,
$$
say. 
Since it is trivial that $|E|\leq (X+1)(Y+1)\ll XY$, 
and since neither the sum $E'$ nor the parameter 
$Z=\rho X^{\gamma} Y^{\delta}$ is changed if one swaps (simultaneously) $X$ and $Y$, 
and $\gamma$ and $\delta$, 
and the relevant sequences $\{\alpha_m\}$ and $\{\beta_n\}$, 
the lemma will therefore follow once it is shown  
that if $\gamma < 1$ then
\begin{equation}\label{desired}
E'\ll 
\left( Z^K X^{\Lambda} Y^{1-K} + X^{\frac12}Y + XY^{\frac12} + Z^{-\frac12}XY\right) \log(2+Y)\;. 
\end{equation} 
We assume, accordingly, that $\gamma <1$. 

If $Y\leq Z$, then by the case $M_2=1$ of \cite[Theorem~2]{B}, it follows that
$$ 
E'\ll XY\log(2+Y)\cdot\left( Y^{-\frac12} + (Z/Y)^{K}X^{\Lambda -1}\right) ,
$$
and so, in this case, we obtain the desired bound \eqref{desired}. 
If instead $Y>Z>1$ then we resort to a method used previously 
in completing the proof of \cite[Theorem~1]{RS}.  
Specifically, using the Bombieri-Iwaniec double large-sieve 
\cite[Lemma~7.5]{GK}, we get 
$$\frac{|E'|^2}{1+Z} \ll
\Biggl( \sum_{\substack{X<m,m_1\leq 2X \\ |m^{\gamma} - m_1^{\gamma}|\, \cdot\, \rho Y^{\delta}\leq 1}} 1\Biggr) 
\Biggl(\sum_{\substack{Y<n,n_1\leq 2Y \\ |n^{\delta} - n_1^{\delta}|\,\cdot\, \rho X^{\gamma}\leq 1}} 1\Biggr) 
\ll X(1+\frac{X}{Z})Y(1+\frac{Y}{Z})\;, 
$$ 
which (using that $Y>Z>1$) gives: $|E'|^2 \ll XY^2 + Z^{-1} X^2 Y^2$ 
(so that \eqref{desired} is obtained in this case also). 
\end{proof} 

Supposing now that $M_4 < \sqrt{N_0}$, we put  
$$
d_3^{\ast}(m) := \left( \frac{N_0}{M_4 m}\right)^{\frac78} 
\cdot\sum_{\substack{m_1m_2m_3 = m \\ M_j<m_j\leq 2M_j\ (j=1,2,3)}} 1 
$$
for $m\in{\mathbb N}$. 
Rewriting the left-hand side of \eqref{s42-1}, 
and using the bounds $0\leq d_3^{\ast}(m) < d_3(m) \ll m^{\varepsilon}$, 
we find that for some $s\in \{ 0,1,2\}$ one has: 
\begin{multline*}
S^{\ast}(M_1,M_2,M_3,M_4)\cdot N_0^{\frac78} \\ 
\begin{split} 
 &= \sum_{\frac{N_0}{M_4} < m\leq \frac{8N_0}{M_4}} 
d_3^{\ast}(m) \sum_{M_4 < n\leq \min\{ 2M_4 , \frac{N}{m}\}} 
M_4^{\frac78} n^{-\frac78}\,e\left( 4x^{\frac14} m^{\frac14} n^{\frac14}\right) \\ 
 &\ll N_0^{\varepsilon} 
\cdot \sum_{m=\bigl[\frac{2^s N_0}{M_4}\bigr] + 1}^{\bigl[\frac{2^{s+1}N_0}{M_4}\bigr]} 
\left|  
\sum_{n=[M_4]+1}^{\min\left\{ 2M_4 ,\bigl[\frac{N}{m}\bigr]\right\}} 
(M_4 / n)^{\frac78}\,e\left( 4x^{\frac14} m^{\frac14} n^{\frac14}\right)\right|\;.  
\end{split} 
\end{multline*}  
Here we recall the exponent pairs $(p,q)$ and $(P,Q)=A(p,q)$     
defined in \eqref{UltimateEP} and \eqref{Def-KL-pair}.   
Since $M_4 < N_0 / M_4$, application of the case 
$\gamma = \delta = \frac14$, $(k,l)=(p,q)$ of Lemma~\ref{P-S_lemma}    
to the above sum over $m$ and $n$ gives us: 
$$
S^{\ast}(M_1,M_2,M_3,M_4)\cdot N_0^{\frac78 - \varepsilon} \\ 
\ll Z^{P} \left( \frac{N_0}{M_4}\right)^{Q} M_4^{1-P}  
+ \frac{N_0}{\sqrt{\min\{ M_4 , Z\} }}\;,
$$ 
where $Z = (2^{s+8} N_0 x)^{1/4}\asymp (N_0 x)^{1/4}$. 
It follows, since $(N_0 x)^{1/4} > N_0^{3/4}\gg N_0^{1/2} > M_4\geq N_0^{1/4}$ 
and (see \eqref{KL-properties}) $1-P-Q\geq 0$, 
that we get   
\begin{align*} 
S^{\ast}(M_1,M_2,M_3,M_4) 
&\ll N_0^{\varepsilon - \frac78} \cdot \left( (N_0 x)^{\frac{P}{4}} N_0^{Q} 
M_4^{1-P-Q}  
+ N_0 M_4^{-\frac12}\right) \\ 
&\ll N_0^{\varepsilon - \frac78} \cdot \left( x^{\frac{P}{4}} 
N_0^{Q + \frac{P}{4} + \frac{1-P-Q}{2}}  
+ N_0^{\frac78}\right) \\ 
&\phantom{\ll {N_0^{\varepsilon - \frac78}}}= N_0^{\varepsilon}\cdot\left( x^{\frac{P}{4}} 
N_0^{\frac{2Q - P - 3/2}{4}} + 1\right) . 
\end{align*}    
Hence, using the first part of \eqref{KL-properties}, we find that 
$$
S^{\ast}(M_1,M_2,M_3,M_4) 
\ll N_0^{\varepsilon} x^{\frac{P}{4}} + N_0^{\varepsilon}\ll x^{\varepsilon + \frac{P}{4}}
\quad\text{if $\,M_4< N_0^{\frac12}$} . 
$$
Combining this conditional bound with the complementary one in \eqref{Type-I-M_4-large}, 
and noting that by \eqref{UltimateEP} and \eqref{Def-KL-pair} we have   
$2P = \frac{\varrho + \varepsilon}{\varrho + \varepsilon + 1} 
< \frac{\varrho}{\varrho + 1} + \varepsilon $,  
we conclude that, regardless of whether or not $M_4 < \sqrt{N_0}$, one has:  
$$
S^{\ast}(M_1,M_2,M_3,M_4) \ll x^{\varepsilon + \frac{\varrho\vphantom{_t}}{8\varrho\vphantom{_t} + 8}}\;.   
$$ 
This, together with the definition \eqref{s4} and 
the bound \eqref{VVV} obtained via our initial splitting argument, 
gives us: 
\begin{equation} \label{s4-final}
S_4 :=\sum_{n \leq N} \frac{d_4(n)}{n^{7/8}}\cos\left(8\pi(nx)^{\frac14}+\frac{7\pi}{4}\right) 
\ll x^{\varepsilon + \frac{\varrho\vphantom{_t}}{8\varrho\vphantom{_t} + 8}}  
\end{equation}
for $x^{1/3}<N\leq x^{1/2}$. 
Substituting \eqref{s4-final} into \eqref{thm2-4} 
and taking (for example) $N=[x^{1/2}]$ and $\eta = \varepsilon / 5$ there, we finally obtain
\begin{equation*}
\sum_{n \leq x}\Delta_4(n) \ll 
x^{\frac98 + \frac{\varrho\vphantom{_t}}{8\varrho\vphantom{_t} + 8} + \varepsilon} \;, 
\end{equation*}
which is the bound \eqref{thm3-4}, the second (final) part of Theorem~\ref{thm3}. \qed 

\begin{remark} \label{TY-EPs}

Substitution of Bourgain's exponent pair 
$\left(\frac{13}{84}+\varepsilon , \frac{55}{84}+\varepsilon \right)$ for 
the exponent pair $(p,q) = \left(\varrho + \varepsilon , \varrho + {\textstyle\frac12} + \varepsilon\right)$ 
converts the above proof of \eqref{thm3-4} 
into a direct proof of \eqref{KL-app-1}. 
Via an elaboration of this direct proof 
that also utilises certain other  exponent pairs, 
we can obtain the bound \eqref{Best-by-newEPs}, which is 
a slight improvement of \eqref{KL-app-1}.  
For this we need, specifically, some of the new (previously unverified) exponent pairs 
discovered by Trudgian and Yang  
in their very interesting recent work \cite{TY} surveying  
the non-trivial exponent pairs that 
are known, or at least knowable, on the basis of current theory alone.  
Before coming to the proof of \eqref{Best-by-newEPs}  
we shall first provide relevant details of the new exponent pairs 
utilised there, starting with their place and 
origins in \cite{TY}. 
\par 
Trudgian and Yang show in \cite[Theorem~1.3]{TY}  
that the closure in ${\mathbb R}^2$    
of the set of non-trivial exponent pairs covered by 
their (apparently comprehensive) survey is the convex hull 
of a certain explicitly defined sequence of points 
$(0,1), (\frac12,\frac12), (k_0,\ell_0), (k_1,\ell_1), (k_{-1},\ell_{-1}),   
(k_2,\ell_2), (k_{-2},\ell_{-2}), \ldots\ $ 
with $(k_{-n},\ell_{-n}) = (\ell_{n} - \frac12, k_{n} +\frac12)$ for $n\in{\mathbb Z}$ 
(exponent pairs $(\kappa,\lambda)$ with  
both $\kappa + \lambda \geq 1$ and $\kappa (\lambda - \frac12)\neq 0$ are considered trivial). 
A part of the proof of \cite[Theorem~1.3]{TY} 
involves showing that, provided $\varepsilon>0$ is small,   
the four points $(k_1+\varepsilon , \ell_1+\varepsilon), \ldots ,(k_4+\varepsilon , \ell_4+\varepsilon)$  
are exponent pairs:   
this is established in  \cite[Section~3]{TY}, 
using results of Sargos, Huxley and Bourgain 
(\cite[Theorem~7.1]{Sa}, \cite[Tables~17.1 and~19.2]{Hu-book} and 
\cite[Eqn.~(3.18)]{Bo}, respectively), as well as
the exponent pair $(k_0+\varepsilon,\ell_0+\varepsilon)$   
(which is, of course, the exponent pair of Bourgain that occurs in \eqref{BourgainEP}).  
Of these four new exponent pairs found by Trudgian and Yang, 
the two most relevant to our present discussion are   
$$ 
{\textstyle \left(\frac{4742}{38463}+\varepsilon,\frac{35731}{51284}+\varepsilon\right) } 
\quad\,\text{and}\quad\, 
{\textstyle\left(\frac{715}{10238}+\varepsilon,\frac{7955}{10238}+\varepsilon\right) } \;, 
$$ 
which are $(k_1+\varepsilon,\ell_1+\varepsilon)$ and $(k_4+\varepsilon,\ell_4+\varepsilon)$, 
respectively: our proof of \eqref{Best-by-newEPs} uses both of these, 
and also Bourgain's exponent pair $(k_0+\varepsilon,\ell_0+\varepsilon)$ 
and the exponent pair $A(k_1+\varepsilon, \ell_1+\varepsilon)$ 
(derived from $(k_1+\varepsilon, \ell_1+\varepsilon)$ via the $A$ process), 
which happens to be 
$$
{\textstyle \left(\frac{2371}{43205}+\varepsilon , \frac{280013}{345640} + \varepsilon\right) } 
=(k_5+\varepsilon,\ell_5+\varepsilon) \;. 
$$
\par 
Our proof of \eqref{Best-by-newEPs} is so similar in form and method 
to the above proof of \eqref{thm3-4} (on which it is modelled) that  
it may be adequately described by mentioning  just a few essential details. 
The key innovations are that, for $M_4\geq \sqrt{N_0}$, we apply \eqref{Type-I-1stUse} with 
$(\kappa,\lambda)$ set equal to either $(k_4+\varepsilon,\ell_4+\varepsilon)$ 
or $(k_5+\varepsilon,\ell_5+\varepsilon)$ 
(which option is best depends on the value that $(\log N_0)/\log x$ takes),   
and that, for $M_4 < \sqrt{N_0}$, depending on the values that $(\log N_0)/\log x$ and 
$(\log M_4)/\log N_0$ take, we either apply the case 
$(\kappa,\lambda)= (k_4+\varepsilon,\ell_4+\varepsilon)$ of \eqref{Type-I}, 
or else use Lemma~6 with $(k,l)$ set equal to either $(k_0+\varepsilon,\ell_0+\varepsilon)$ 
or $(k_1+\varepsilon,\ell_1+\varepsilon)$. 
In the most critical cases, where 
$$
N_0\asymp x^{\frac{13689}{73790}} = x^{\frac{(3\times 39)^2}{73790}} \quad\,\text{and}\quad\,  
M_4^2\asymp  N_0^{1 - \frac{1}{39^2}}\asymp x^{\frac{3^2\times (39^2 - 1)}{73790}} = x^{\frac{1368}{7379}}\;,  
$$ 
the three options just outlined  
all yield the same bound for the modulus of the sum $S^{\ast}(M_1,M_2,M_3,M_4)$ 
defined in \eqref{s42-1}, which is: 
$$
S^{\ast}(M_1,M_2,M_3,M_4)\ll x^{\frac{2471}{147580}+\varepsilon} \;. 
$$
In all other cases the bound obtained for $|S^{\ast}(M_1,M_2,M_3,M_4)|$ is no worse than this. 
It follows, by \eqref{VVV}, that  
we can replace the exponent of $x$ on the right-hand side 
of \eqref{s4-final} with $\frac{2471}{147580}+\varepsilon$, 
so that (arguing as at the end of the above proof of \eqref{thm3-4}) we obtain the 
bound \eqref{Best-by-newEPs}. 
Similarly to how Corollary~\ref{cor2} is proved (see below), we 
get from \eqref{Best-by-newEPs} the corollary that 
$\int_1^x \Delta_4(t)dt\ll x^{\frac98 + \frac{2471}{147580}+\varepsilon}$.  

\end{remark}


\section{Proving Corollary \ref{cor2} -- with Segal's identity, and without it} 

\begin{proof}[Proof of Corollary \ref{cor2}] 
Let $k=3$ or $4$.   Since $\Delta_k(u)$ is right-continuous and has derivative
$\Delta_k'(v)=-(P_k(\log v)+P_k'(\log v)) \ll 1+\log^{k-1}v$ at all non-integer points $v>1$, we have
$\Delta_k(u)-\Delta_k([u]) \ll (u-[u])(1+\log^{k-1} u) \ll 1+\log^{k-1}u$ for $u \geq 1$, so that
\begin{align}
\int_1^x \Delta_k(u)du &= \int_1^x \Delta_k([u])du+O\left(\int_1^x(1+\log^{k-1}u)du \right) \label{eq-91}  \\
&=\sum_{n \leq x}\Delta_k(n)+O(|\Delta_k(x)|)+O\left(1+x \log ^{k-1}x\right) \notag
\end{align}
for $x \geq 1$. 
Using \eqref{eq-91} and \eqref{D34-bound}, we find that Corollary \ref{cor2} follows from Theorem \ref{thm3}. 
\end{proof} 

Here we have to mention Segal's identity which connects the sum of $\Delta_k(n)$ and the integral of $\Delta_k(x)$.
\begin{lemma}[Segal \cite{S}] \label{Segal}
Let $f(n)$ be a function of a positive integral variable, and suppose 
$\sum_{n \leq x} f(n)=g(x)+E(x),$ where $g(x)$ is twice continuously differentiable, and 
$g''(x)$ is of constant sign for $x \geq 1$. Then 
\begin{equation*} 
\sum_{n \leq x} E(n)=\frac12 g(x)+(1-\{x\})E(x)+\int_1^x E(t)dt+O(|g'(x)|)+O(1),
\end{equation*}
where $\{x\}=x-[x]$ indicates the fractional part of $x$. 
\end{lemma}

In his proof of this Segal first transformed the sum $\sum_{n\leq x} E(n)$ into an 
expression involving the sums $\sum_{n\leq x} f(n)$, $\sum_{n\leq x} nf(n)$ and $\sum_{n\leq x} g(n)$. 
He dealt with the second last and last of these sums using, respectively,  
partial summation and the Euler-Maclaurin summation formula. 
Segal's identity was generalized to a relation between the sum and the integral of higher powers 
of $E(x)$ by Furuya \cite{F}. Furuya \cite{F} and Cao, Furuya, Tanigawa and Zhai \cite{CFTZ} gave 
detailed expressions for the difference between the continuous mean (integral) and 
discrete mean (sum) of $\Delta_2(x)$ and its powers.

In our case $f(n)=d_k(n)\ (k=3, 4)$, and we have $g(x)=xP_k(\log x)$ and $E(x)=\Delta_k(x)$.
From \eqref{shukou} and \eqref{shukou4}, we see that 
\begin{align*}
g''(x)&=\frac{1}{x}(P_k'(\log x)+P_k''(\log x))  \\
      &=\begin{cases} \d \frac{1}{x}\left(\log x+3\gamma_0\right) & k=3, \\[1ex]
              \d \frac{1}{x} \left(\frac12 (\log x)^2 + 4 \gamma_0 \log x   
              + 4 \gamma_1+6\gamma_0^2\right) & k=4.
         \end{cases}
\end{align*}  
Since $\gamma_0=0.577\ldots$ and $\gamma_1=0.07281\ldots$, 
the assumptions of Lemma \ref{Segal} are satisfied for $k=3, 4$.
Therefore 
Lemma \ref{Segal} can be used to show that Corollary \ref{cor2} follows 
immediately from Theorem \ref{thm3}, and vice versa. 
We can also give a direct proof of Corollary \ref{cor2}, using Lemma \ref{voronoitype} as follows.

Let $k$ be a positive integer (only later do we assume $k=3$ or $4$). 
By the definition of $\Delta_k(x)$ we have
\begin{equation} \label{9Aa} 
\int_0^x \Delta_k(u)du=\int_0^x\left(\sum_{n \leq u}d_k(n)\right)du-\int_0^x uP_k(\log u)du.
\end{equation} 
Theorem II.2.5 of Tenenbaum \cite{Te} with $a_n=d_k(n)$ implies that 
\begin{equation} \label{9Bb} 
\int_0^x \left(\sum_{n \leq u}d_k(n)\right)du 
= \frac{1}{2\pi i} \int_{(c)} \zeta^k(s) \frac{x^{s+1}}{s(s+1)}ds,
\end{equation} 
where $c$ is an arbitrary constant greater than $1$. Let $T$ be a large positive number.  
By Cauchy's theorem we have
\begin{multline*} 
\frac{1}{2\pi i}\int_{c-iT}^{c+iT}\zeta^k(s)\frac{x^{s+1}}{s(s+1)}ds  \\ 
\begin{split} 
 &=\frac{1}{2\pi i}\int_{\frac12-iT}^{\frac12+iT} 
\zeta^k(s)\frac{x^{s+1}}{s(s+1)}ds+\Res_{s=1}\left(\zeta^k(s)\frac{x^{s+1}}{s(s+1)}\right) \\
&{\phantom{=}} 
+\frac{1}{2\pi i}\left(\int_{c-iT}^{\frac12-iT}+\int_{\frac12+iT}^{c+iT}\right)\zeta^k(s)\frac{x^{s+1}}{s(s+1)}ds.
\end{split}  
\end{multline*} 
The last two integrals above are equal in modulus (they are $iz$ and $i\overline{z}$, for some $z\in{\mathbb C}$).   
Using the estimate $\zeta(\sigma+iT) \ll T^{\frac{c-\sigma}{3(2c-1)}}$ for $1/2 \leq \sigma \leq c$, we get 
\begin{align*}
\int_{\frac12+iT}^{c+iT}\zeta^k(s)\frac{x^{s+1}}{s(s+1)}ds 
& \ll T^{\frac{kc}{3(2c-1)}-2}x \int_{\frac12}^{c}\left(\frac{x}{T^{\frac{k}{3(2c-1)}}}\right)^{\sigma}  d\sigma \\
& \ll x^{\frac32}T^{\frac{k}{6}-2}\left(\log\frac{T^{\frac{k}{3(2c-1)}}}{x}\right)^{-1}
\end{align*}
for $T^{\frac{k}{3(2c-1)}} > x$. Hence for $k \leq 12$, these integrals along the horizontal line segments  
tend to $0$ when $T \to \infty$. Thus, recalling \eqref{9Aa} and \eqref{9Bb}, we get  
\begin{align}
\int_0^x \Delta_k(u)du&=\frac{1}{2\pi i} \int_{(\frac12)} \zeta^k(s) \frac{x^{s+1}}{s(s+1)}ds
+\Res_{s=1}\left(\zeta^k(s)\frac{x^{s+1}}{s(s+1)}\right) \label{8-01}   \\
& \quad  -\int_0^x uP_k(\log u)du ,  \nonumber
\end{align}
provided $k\leq 12$. The second and third terms on the right-hand side of \eqref{8-01} cancel out. 
In fact by \eqref{Pk} we have
$$
\int_0^x uP_k(\log u)du=\frac{1}{2\pi i} \int_0^x \left(\int_{C(1, r)}\zeta^k(s)\frac{u^s}{s}ds\right) du,
$$
where $C(1,r)$ is a circle of radius $r < 1/4$ with center $1$. 
And by inverting the order of integration in the latter (iterated) integral, we get 
\begin{equation*} 
\int_0^x uP_k(\log u)du=\frac{1}{2\pi i} \int_{C(1, r)}\zeta^k(s)\frac{x^{s+1}}{s(s+1)}ds
=\Res_{s=1}\left(\zeta^k(s)\frac{x^{s+1}}{s(s+1)}\right). 
\end{equation*} 
Hence (using also that $\int_0^1 \Delta_k(u)du\ll1$) Equation \eqref{8-01} simplifies to: 
\begin{align*} 
\int_{1}^{x}\Delta_k(u)du =\frac{1}{2\pi i}\int_{(1/2)}\zeta^k(s)\frac{x^{s+1}}{s(s+1)}ds+O(1).
\end{align*}
We have shown this holds provided $k\leq 12$. 

Next, we truncate  the above integral at the points $s=\frac12 \pm iT$, 
with $T=2\pi(x(N+\frac12))^{\frac1k}$, as in Lemma \ref{voronoitype}. 
Assuming now that $k=3$ or $k=4$, we estimate the truncation error using  
$\int_{1/2+iT}^{1/2+i\infty}|\zeta^k(s)/s(s+1)||ds| \ll T^{-1+\varepsilon}$, 
which is a valid bound in these two cases. Hence 
\begin{equation} \label{8-1}
\int_{1}^{x}\Delta_k(u)du=\frac{1}{2\pi i}\int_{1/2-iT}^{1/2+iT}\zeta^k(s)\frac{x^{s+1}}{s(s+1)}ds
+O\left(x^{3/2}T^{-1+\varepsilon}\right)+O(1).
\end{equation}
As we did in \eqref{U4sekibun}--\eqref{U4},  we now (using Cauchy's theorem) shift  the integral 
on the right-hand side of \eqref{8-1}  to $\int_{-\delta-iT}^{-\delta+iT}$,  
where $\delta$ is taken to be  as in Lemma \ref{voronoitype}.  This gives us: 
\begin{align} \label{8-2}
\int_{1}^{x}\Delta_k(u)du&=\frac{1}{2\pi i}\int_{-\delta-iT}^{-\delta+iT}\zeta^k(s)\frac{x^{s+1}}{s(s+1)}ds
+\left(-\frac12\right)^k x+O\left(x^{3/2}T^{k/6-2}\right)  \\[1ex]
& \quad +O\left(x^{1-\delta}T^{k(1/2+\delta)-2}\right) +  O\left(x^{3/2}T^{-1+\varepsilon}\right) +O(1) \notag 
\end{align}
(with the first two $O$-terms here coming from the relevant integrals over horizontal line segments). 
Application of Lemma \ref{voronoitype} to the right-hand side of \eqref{8-2} yields 
\begin{align} \label{eq-9}
\int_1^x \Delta_k(u)du=V_k(x,N)+O\left(\frac{x^{\frac32+\varepsilon}}{(xN)^{\frac{1}{k}}}\right)+
\begin{cases} -\frac{1}{8}x+O\left(x^{\frac{5}{6}}\right), &  \text{if $k=3$}, \\
              O\left(xN^{\delta}\right), & \text{if $k=4$},
\end{cases}
\end{align}
where $V_3(x,N)$ and $V_4(x,N)$ are the functions defined by \eqref{V3} and \eqref{V4}, respectively.  
The sums over $n$ that occur in \eqref{V3} and \eqref{V4} can be estimated as in the proof of Theorem \ref{thm3},
taking $N=[x^{1/(k-2)}]$. Substituting the resulting bounds for $V_3(x,N)$ and $V_4(x,N)$ into \eqref{eq-9}, 
and taking $\delta=\varepsilon$ in the case $k=4$, and $\varepsilon\leq\frac{3}{20}$  in either case, 
we get Corollary \ref{cor2}.

\begin{remark} \label{Rem-msq2}
Squaring the right-hand sides of \eqref{eq-9} and using \eqref{msq-3}, \eqref{msq-4} and Cauchy's inequality, and 
taking $N=[X]$ in both cases, we get
\begin{align*} 
\int_1^X\left(\int_1^x \Delta_3(u)du\right)^2dx&=\left(\frac{1}{72\pi^4}\sum_{n=1}^{\infty}\frac{d_3^2(n)}{n^2}+\frac{1}{192}\right)X^3
+O(X^{\frac{17}{6}+\varepsilon})    
\intertext{and} 
\int_1^X\left(\int_1^x \Delta_4(u)du\right)^2dx&=\frac{1}{104\pi^4}\left(\sum_{n=1}^{\infty}\frac{d_4^2(n)}{n^{7/4}}\right)X^{\frac{13}{4}}
+O(X^{\frac{25}{8}+\varepsilon}) 
\end{align*} 
(we use here the estimate for the integral $\int_1^X xV_3(x,N)dx$ that the first derivative test yields). 
These last two asymptotic formulae, each containing an explicit main term and (smaller) $O$-term, 
are qualitative improvements of the corresponding upper bounds, respectively 
$O(X^{3+\varepsilon})$ and $O(X^{\frac{13}{4}+\varepsilon})$, 
that were obtained by Ivi\'c  \cite[Theorem~2]{I2}.   
We remark that although the `$c_4=\varepsilon$' in \cite[p.~37, l.~4]{I2} needs correction, 
the connected result $\rho_4\leq\frac{13}{4}$ in \cite[Theorem~2]{I2} comes from 
(correctly) having $c_4 = \frac18 +\varepsilon$. 
\end{remark}

Using \eqref{eq-91} (or Lemma \ref{Segal}) and \eqref{8-1}, one can connect the sum   
$\sum_{n \leq x}\Delta_k(n)$ with the truncated integral 
$\frac{1}{2\pi i}\int_{1/2-iT}^{1/2+iT} \zeta^k(s)\frac{x^s}{s(s+1)} ds$.    
This connection is however somewhat indirect, since involves (as a first step) 
the approximation of the sum by the corresponding integral: $\int_1^x \Delta_k(u)du$.     
The connection becomes a more direct one when one uses instead the quite different method 
that we have developed in Sections 2--5. Though this more direct method is a little complicated, 
we think that it is new, and that it will have applications in studying sums of error terms from  
other interesting arithmetical problems. It can be used even where Segal's identity 
(Lemma \ref{Segal}) does not apply: it shall, for example, be employed in 
the the first and second author's future study \cite{MT} of
the summatory function of the error term from the asymptotic formula for Ingham's classical sum 
$\sum_{m=1}^{n-1} d(m)d(n-m)$.

\newpage

\noindent T. Makoto Minamide\\
Graduate School of Sciences and Technology for Innovation\\
Yamaguchi University\\
Yoshida 1677-1, Yamaguchi 753-8512, Japan\\
e-mail: minamide@yamaguchi-u.ac.jp\\

\noindent Yoshio Tanigawa\\
Nagoya, Japan \\
e-mail: tanigawa@math.nagoya-u.ac.jp \\

\noindent Nigel Watt\\
Dunfermline, Scotland\\
e-mail: wattn@btinternet.com

\end{document}